\documentclass[12pt]{amsart}
\usepackage{amssymb,latexsym,comment,url}
\usepackage{amsrefs}
\usepackage[all]{xy}
\usepackage{color}
\usepackage{mathrsfs}
\usepackage[hidelinks, colorlinks=true, linkcolor=blue, citecolor=magenta]{hyperref}

\newcommand{\wB}{B} 

\newcommand{\C}{{\mathbb C}}

\newcommand{\wA}{A} 
\newcommand{\Br}{{\operatorname{Br}}}
\newcommand{\D}{{\mathcal{D}}}
\newcommand{\ASL}{\operatorname{ASL}}

\newcommand{\Char}{\operatorname{char}}

\newcommand{\Gal}{\operatorname{Gal}}
\newcommand{\NS}{\operatorname{NS}}
\newcommand{\Hom}{\operatorname{Hom}}

\newcommand{\Pic}{\operatorname{Pic}}

\newcommand{\Lbar}{\overline{L}}

\newcommand{\G}{{\mathcal G}}

\newcommand{\GL}{\operatorname{GL}}

\newcommand{\isom}{{\, \cong \,}}
\newcommand{\la}{\lambda}

\newcommand{\K}{k}
\newcommand{\Kbar}{\K^\sep}

\newcommand{\Map}{\operatorname{Map}}

\newcommand{\PP}{{\mathbb P}}
\newcommand{\Q}{{\mathbb Q}}

\newcommand{\rank}{\operatorname{rank}}
\newcommand{\ra}{{\longrightarrow}}

\newcommand{\sep}{\mathrm{sep}}
\newcommand{\SL}{\operatorname{SL}}

\newcommand{\tr}{{\text{\rm tr}}}

\newcommand{\LL}{{\mathcal{L}}}

\newcommand{\x}{{\bf x}}

\newcommand{\Z}{{\mathbb Z}}
\newcommand{\Aut}{\operatorname{Aut}}
\newcommand{\adj}{\operatorname{adj}}
\newcommand{\ob}{\operatorname{Ob}}
\newcommand{\Gm}{\mathbb{G}_{\text{\rm m}}}
\DeclareFontEncoding{OT2}{}{} 
\newcommand{\textcyr}[1]{%
 {\fontencoding{OT2}\fontfamily{wncyr}\fontseries{m}\fontshape{n}\selectfont #1}}
\newcommand{\Sha}{{\mbox{\textcyr{Sh}}}}

\newtheorem{Proposition}{Proposition}[section]
\newtheorem{Theorem}[Proposition]{Theorem}
\newtheorem{Lemma}[Proposition]{Lemma}
\newtheorem{Corollary}[Proposition]{Corollary}

\theoremstyle{definition}
\newtheorem{Definition}[Proposition]{Definition}
\newtheorem{Remark}[Proposition]{Remark}

\newtheorem{Example}[Proposition]{Example}

\DeclareMathOperator{\End}{End}
\begin{document}

\title[Visibility of 4-covers of elliptic curves]%
{Visibility of 4-covers of elliptic curves}
\author{Nils~Bruin}
\address{Department of Mathematics, Simon Fraser University, Burnaby, BC, V5A 1S6, Canada}
\email{nbruin@sfu.ca}

\author{Tom~Fisher}
\address{University of Cambridge,
         DPMMS, Centre for Mathematical Sciences,
         Wilberforce Road, Cambridge CB3 0WB, UK}
\email{T.A.Fisher@dpmms.cam.ac.uk}

\date{23rd January 2017}  

\begin{abstract}
  Let $C$ be a $4$-cover of an elliptic curve $E$, written as a
  quadric intersection in $\PP^3$. Let $E'$ be another elliptic curve
  with $4$-torsion isomorphic to that of $E$. We show how to write
  down the $4$-cover $C'$ of $E'$ with the property that $C$ and $C'$
  are represented by the same cohomology class on the $4$-torsion.  In
  fact we give equations for $C'$ as a curve of degree $8$ in $\PP^5$.

  We also study the K3-surfaces fibred by the curves $C'$ as we vary
  $E'$.  In particular we show how to write down models for these
  surfaces as complete intersections of quadrics in $\PP^5$ with
  exactly $16$ singular points.  This allows us to give examples of
  elliptic curves over $\Q$ that have elements of order $4$ in their
  Tate-Shafarevich group that are not visible in a principally
  polarized abelian surface.
\end{abstract}

\maketitle

\renewcommand{\theenumi}{\roman{enumi}}

\section{Introduction}
\label{intro}
Let $E$ and $E'$ be elliptic curves over a field $k$ that are
\emph{$n$-congruent}, meaning that there is an isomorphism of $k$-group
schemes $\sigma\colon E[n]\to E'[n]$. We suppose that the
characteristic of $k$ does not divide $n$. We may use $\sigma$ to
transfer certain interesting arithmetic information between $E$ and
$E'$. For instance let $k$ be a number field. An $n$-torsion element
of the \emph{Tate-Shafarevich} group $\Sha(E/k)$ may be represented by
a class $\xi\in H^1(k,E[n])$.  Let $\sigma_* : H^1(k,E[n]) \to
H^1(k,E'[n])$ be the isomorphism induced by $\sigma$. It might happen
that while $\xi$ maps to a non-trivial element in $\Sha(E/k)[n]$,
represented say by a curve $C$, the image of $\sigma_*(\xi)$ in $
H^1(k,E')$, represented say by a curve $C'$, could be trivial. This is
an example of Mazur's concept of \emph{visibility} (see
\cites{CM,Mazur-cubics}): the graph $\Delta\subset E[n]\times E'[n]$
of the $n$-congruence $\sigma$ provides an isogeny $E\times E'\to
(E\times E')/\Delta$ and a model for $C$ arises as the fibre of
$(E\times E')/\Delta$ over a point in $E'(k)$ that bears witness to
the triviality of $C'$.

The case $n=2$ is relatively special, because quadratic twists have
isomorphic $2$-torsion. It is true, however, that given any $\xi\in
H^1(k,E[2])$, one can find another elliptic curve $E'$ with isomorphic
$2$-torsion such that $\sigma_*(\xi) \in H^1(k,E'[2])$ represents a
trivial homogeneous space under $E'$; see \cites{bruin:sha2,
  Klenke2005}.

In order to determine if a given class $\xi\in H^1(k,E[n])$ can be
made visible using an $n$-congruence, one can proceed in three
steps. We assume $n>2$.
\begin{enumerate}
\item One parametrizes the elliptic curves $n$-congruent to $E$. This
  amounts to determining an appropriate twist $X_E(n)$ of the modular
  curve
  of full level $n$. (We make no reference to the Weil
  pairing in the definition of
  $X_E(n)$, so geometrically
  this curve has $\phi(n)$ components, where $\phi$ is Euler's
  totient function.)
\item If $n>2$ then $X_E(n)$ is a fine moduli space and there is a
  universal elliptic curve $E_t$ over $X_E(n)$.  One constructs a
  fibred surface $S_{E,\xi}(n) \to X_E(n)$ whose fibres are the
  $n$-covers of $E_t$
  corresponding to $\xi \in H^1(k,E_t[n])$.
\item If one can find a rational point $P$ on $S_{E,\xi}(n)$, and none
  of the cusps of $X_E(n)$ are rational points,
  then $\xi$ can be made visible by taking $E'$ to be the elliptic
  curve corresponding to the moduli point on $X_E(n)$ below $P$.  On
  the other hand, if
  $S_{E,\xi}(n)$ has no rational points then $\xi$ cannot be made
  visible using an elliptic curve $n$-congruent to $E$.
\end{enumerate}

One can classify the $n$-congruence $\sigma$ by the effect it has on
the Weil pairing. If $n=3,4$, it can either preserve or invert
it. Correspondingly, the modular curve $X_E(n)$ has two components
$X^+_E(n)$ and $X^-_E(n)$. Note that there is a tautological point on
$X^+_E(n)$ corresponding to $E$ itself, and the fibre of
$S_{E,\xi}^+(n) \to X^+_E(n)$ above this point represents the image of
$\xi$ in $H^1(k,E)$. If this image lies in $\Sha(E/k)$ then the fibre,
and hence also $S_{E,\xi}^+(n)$ itself, has points everywhere locally.
Mazur uses this in \cite{Mazur-cubics} to show that any element of
$\Sha(E/k)$ of order $3$ can be made visible using an elliptic curve
$3$-congruent to $E$. Indeed he shows that $S^+_{E,\xi}(3)$ is a
blow-up of a twist of $\PP^2$, and hence satisfies the local-to-global
principle.

From a computational point of view, it is attractive if $(E\times
E')/\Delta$ can be realized as a Jacobian, or more generally, admits a
principal polarization. It naturally does so if we start with $\sigma$
inverting the Weil pairing. Again taking $n=3$, one can show that
$S^+_{E,\xi}(3)$ is birational to $\PP^2$ over $k$ if and only if the
same is true for $S^-_{E,\xi}(3)$. It follows (see
\cite{BruinDahmen2010}) that any element of $\Sha(E/k)$ of order $3$
is visible in the Jacobian of a genus $2$ curve.

For larger $n$ there are major obstacles to this kind of visibility
over number fields. Once $n\geq 6$ the components of $X_E(n)$ have positive
genus, and so rational points are rare: the set of candidate elliptic
curves $E'$ is sparse for $n=6$ and finite for $n\geq 7$. See
\cite{fisher:invis}  for explicit examples over $\Q$ of non-existence
of such $E'$ for $n=6,7$.

In this article we consider the case $n=4$. This case is particularly
interesting for several reasons:
\begin{enumerate}
\item The curve $X_E^+(4)$ is of genus $0$, but the surface
  $S^+_{E,\xi}(4)$ is a K3-surface. Much is conjectured, but little is
  known about the rational points on K3-surfaces.
\item Given $\xi \in H^1(k,E[n])$, representing a genus $1$ curve of
  degree $n$, there is another fibred surface $T_{E,\xi}(n) \to
  X_E(n)$ whose fibres are $n$-covers of $E_t$, but now sharing the
  same action of $E[n]$ on a suitable linear system.  For $n=3,4,5$,
  explicit invariant-theoretic constructions of these surfaces are
  given in \cite{g1hessians}, \cite{enqI}. When $n$ is odd the
  surfaces $S^\pm_{E,\xi}(n)$ and $T^\pm_{E,\xi}(n)$ are the same, but when
  $n$ is even a correction to this idea is needed.
\end{enumerate}

Taking $n=4$, we may identify $X_E^{\pm}(4) \isom \PP^1$.  We start
with $D = \{ Q_1 = Q_2 = 0 \} \subset \PP^3$ a quadric intersection
representing $\xi \in H^1(k,E[4])$. The invariant theory in
\cite{g1hessians} allows us to write down
$T^\pm_{E,\xi}(4) \to \PP^1$ as a family of quadric intersections in
$\PP^3$. However the fibres of $S^\pm_{E,\xi}(4) \to \PP^1$ cannot
always be written as quadric intersections. In this article we show
how to write down a singular model for $S^\pm_{E,\xi}(4)$ as a
complete intersection of quadrics in $\PP^5$. On this model, the
non-singular fibres
are genus $1$ curves of degree $8$.

In fact the surfaces $S^\pm_{E,\xi}(4)$ and $T^\pm_{E,\xi}(4)$ are
twists of surfaces $S(4)$ and $T(4)$ that may be defined as
follows. 
Let $X(n)$ be the modular curve whose non-cuspidal points
parametrise elliptic curves $E$ together with a symplectic isomorphism
$E[n] \isom \Z/n\Z \times \mu_n$. We write
\[E_t\colon\quad y^2=x(x-1)(x-\frac{(1-t^2)^2}{(1+t^2)^2})\] for the
universal elliptic curve over $X(4) \isom \PP^1$. The {\em Shioda
  modular surface} of level $4$ is the minimal fibred surface $S(4)
\to X(4)$ with generic fibre $E_t$. The {\em theta modular surface} of
level $4$ is the minimal fibred surface $T(4) \to X(4)$ with generic
fibre
\[
  D_t\colon\quad\left\{
\begin{aligned}
t(x_0^2 + x_2^2) + 2 x_1 x_3 &= 0 \\
t(x_1^2 + x_3^2) + 2 x_0 x_2 &= 0
\end{aligned} \right\} \subset \PP^3. 
\]
The relation between the fibres is that $D_t$ has Jacobian $E_t$.

We refer to \cite{BH} for many interesting facts about the geometry of
the surfaces $S(4)$ and $T(4)$. For example, they are both K3-surfaces
(in fact Kummer surfaces) with Picard number 20. Working over $\C$,
the surface $T(4)$ is isomorphic to the diagonal quartic surface in
$\PP^3$.  Moreover the surfaces $S(4)$ and $T(4)$ are related by
generically $2$-to-$1$ rational maps (in either direction) but are not
birational.

\medskip

\subsection{Outline of the article}

In Section~\ref{S:preliminaries} we review some of the interpretations
of $H^1(k,E[n])$, most notably in terms of $n$-covers and theta
groups. We also define the (twisted) Shioda and theta modular
surfaces.

In Section~\ref{sec:4cover} we look at methods for computing
$4$-covers of elliptic curves.  A central notion is that of a
\emph{second $2$-cover}: any $4$-cover $D\to E$ factors through a
$2$-cover $C\to E$. The cover $D\to C$ is referred to as a
\emph{second $2$-cover}. In terms of Galois cohomology this
corresponds to the sequence
\[H^1(k,E[2])\to H^1(k,E[4])\to H^1(k,E[2]),\] where the first
$H^1(k,E[2])$ classifies the cover $D\to C$ and the second $H^1(k,E[2])$
classifies $C\to E$. 

In Section~\ref{S:24trivial} we review classical $4$-descent. It has
the drawback for us that it only gives those $4$-covers with a degree
$4$ model in $\PP^3$.  In Section~\ref{S:Dprime_nu} we describe a
variant of this method.  In particular, given $D\to C\to E$ where $D$
has a degree~$4$ model in $\PP^3$, we show how to twist the second
$2$-cover $D \to C$ by an arbitrary element of $H^1(k,E[2])$. The new
$4$-cover has a degree $8$ model in $\PP^7$.
However, for our purposes, it is convenient to project this to a curve
in $\PP^5$, still of degree $8$.
	
In Section~\ref{sec:thetashift} we quantify the difference between
$S^\pm_{E,\xi}(n)$ and $T^\pm_{E,\xi}(n)$ for arbitrary $n$.  Indeed
the fibres differ by a cohomology class $\nu = \nu(t)$, which we call
the {\em shift}.  It was already shown in \cite{descI}*{Lemma 3.11}
that the shift is trivial when $n$ is odd. We show that when $n$ is
even the shift takes values in $H^1(\K,E[2])$.

Section~\ref{sec:shioda} reviews the geometry of the surfaces $S(4)$
and $T(4)$.  In Section~\ref{S:twist_shioda} we describe how to
compute the twists of these surfaces so that a prescribed $4$-cover $D
\to E$ appears as a fibre. We assume that $D$ is given as a quadric
intersection in $\PP^3$.  We use the invariant theory
in~\cite{g1hessians} to write down the required twist of $T(4)$.  By
finding an explicit formula for the shift, and then using the method
in Section~\ref{S:Dprime_nu}, we are then able to compute the required
twist of $S(4)$.

The methods in Section~\ref{S:twist_shioda} for computing
$S^+_{E,\xi}(4)$ and $T^+_{E,\xi}(4)$ are modified in
Section~\ref{S:rev} to compute $S^-_{E,\xi}(4)$ and $T^-_{E,\xi}(4)$.
The arguments here are somewhat simplified by the observation that an
elliptic curve and its quadratic twist by its discriminant are reverse
$4$-congruent.

In Section~\ref{S:examples} we give several examples.
We exhibit some elliptic curves $E/\Q$ such that for the elements $\xi\in H^1(\Q,E[4])$ representing elements of order $4$ in $\Sha(E/\Q)$, the surface $S^-_{E,\xi}(4)$ has no rational points. We show this by computing an explicit model of the surface and checking that the surface has no $p$-adic points for some prime $p$. If $E(\Q)/2E(\Q)$ is trivial, it follows that visibility in a surface can only happen via a rational point on $S^+_{E,\xi}(4)$. We prove in Proposition~\ref{P:noprincpol} that if the Galois action on $E[4]$ is large enough, then the resulting abelian surface does not admit a principal polarization.
	
We note that if $\xi \in H^1(\K,E[4])$ represents an element in
$\Sha(E/\K)$ then $S^+_{E,\xi}(4)$ has points everywhere
locally. Thus, any failure for $S^+_{E,\xi}(4)$ to have rational
points constitutes a failure of the Hasse Principle. We plan to
investigate this possibility further in future work.

\section{Preliminaries}
\label{S:preliminaries}

\subsection{Notation}

For a field $\K$, we write $\Kbar$ for its separable closure.  We
write $t$ for the generic point on various modular curves we
consider. When these curves are isomorphic to $\PP^1$, then $t$ will
instead denote a coordinate on $\PP^1$. The identity element on an
elliptic curve $E$ will be written as either $0$ or $0_E$.

\subsection{Geometric interpretations of $H^1(k,E[n])$}

Let $k$ be a field of characteristic not dividing $n$ and let $E$ be
an elliptic curve over $k$.

\begin{Definition}
  An \emph{$n$-cover} of $E$ over $k$ is a pair $(C,\pi)$, where $C$
  is a nonsingular complete irreducible curve over $k$ and $\pi\colon
  C\to E$ is a morphism such that there exists an isomorphism
  $\psi\colon (C\times_\K\Kbar)\to (E\times_\K\Kbar)$ satisfying
  $\pi=[n]\circ \psi$. Two $n$-covers $(C_1, \pi_1)$ and $(C_2,\pi_2)$
  are \emph{isomorphic over $k$} if there is an isomorphism $\alpha :
  C_1 \to C_2$ over $k$ such that $\pi_1 = \pi_2 \circ \alpha$.
\end{Definition}

The isomorphism classes of $n$-covers are naturally parametrized by
$H^1(k,E[n])$. This means that given $\xi\in H^1(k,E[n])$ there is an
$n$-cover $C_{E,\xi}\to E$, any $n$-cover is isomorphic to one of this
form, and two $n$-covers are isomorphic if and only if they arise from
the same class $\xi$. Restriction of cocycle classes corresponds to
base extending the cover.

If $(C,\pi)$ is an $n$-cover of $E$ then
$C$ itself is a twist of $E$ (as a curve, not as an elliptic
curve). In fact $C$ has the structure of homogeneous space under $E$,
and so represents an element in $H^1(k,E)$. The map $H^1(k,E[n])\to
H^1(k,E)$ may be interpretted as forgetting the covering map $\pi$. In
particular its kernel consists of those $n$-covers $(C,\pi)$ for which
$C(k)$ is non-empty.

\begin{Definition}
  A \emph{theta group for $E[n]$} is a central extension $0\to \Gm\to
  \Theta\to E[n]\to 0$ of $k$-group schemes such that the commutator
  pairing on $\Theta$ agrees with the Weil-pairing on $E[n]$. An
  isomorphism of theta groups is an isomorphism of central extensions
  as $k$-group schemes.
\end{Definition}

The lift of the (translation) action of $E[n]$ on the linear system
$|n.0_E|$ to the Riemann-Roch space $L(n.0_E)$ gives a theta group
$\Theta_E$. If $n\geq 3$ then $L(n.0_E)$ is the space of
global sections of a very ample line bundle $\LL_{E,n}$. Choosing a
basis for this space provides a map $E\to\PP^{n-1}$ that gives a model
for $E$ as an elliptic normal curve of degree $n$.  The theta group
$\Theta_E$ is then the full inverse image in $\GL_n$ of the group of
projective linear transformations describing the action of $E[n]$ on
$E$ by translation.

As observed in \cite{descI}, there is an action of $E[n]$ on
$\Theta_E$ by conjugation, and every automorphism of $\Theta_E$ arises
in this way.  Therefore the isomorphism classes of theta groups for
$E[n]$, viewed as twists of $\Theta_E$, are parametrised by
$H^1(k,E[n])$.

Since an $n$-cover $C_{E,\xi}$ (as a curve) is a twist of $E$ by a
cocycle taking values in $E[n]$, we see that $C_{E,\xi}$ comes
equipped with a degree $n$ line bundle $\LL_{E,\xi}$ with a theta
group $\Theta_{E,\xi}$ acting on it. It may be checked that
$\Theta_{E,\xi}$ is indeed the twist of $\Theta_E$ by $\xi$ (in the
sense of the last paragraph).  The line bundle $\LL_{E,\xi}$ provides
a model of $C_{E,\xi}$, but now only in an $(n-1)$-dimensional
Brauer-Severi variety, i.e. a possibly non-trivial twist of
$\PP^{n-1}$. The $k$-isomorphism class of the Brauer-Severi variety
gives a class in $\Br(k)[n]$.

\begin{Definition}
  We write $\ob_{E,n}\colon H^1(k,E[n])\to\Br(k) $ for the map that
  sends $\xi\in H^1(k,E[n])$ to the class of the Brauer-Severi variety
  corresponding to the global sections of $\LL_{E,\xi}$.  In
  particular $\ob_{E,n}(\xi)$ is trivial if and only if $C_{E,\xi}$
  admits a degree $n$ model in $\PP^{n-1}$ with $\LL_{E,\xi}$ the pull
  back of ${\mathcal O}(1)$.  In later sections we write
  $\ob_n(C_{E,\xi})=\ob_{E,n}(\xi)$.
\end{Definition}

It is shown in \cite{descI} that $\ob_{E,n}(\xi)$ is determined by the
isomorphism class of $\Theta_{E,\xi}$ (as a theta group for $E[n]$)
without reference to $E$ itself.

\subsection{Twists of full level modular curves}

An \emph{$n$-congruence} between two elliptic curves $E,E'$ over $k$
is an isomorphism of $k$-group schemes $\sigma\colon E[n]\to
E'[n]$. We only concern ourselves with the case that the
characteristic of $k$ does not divide $n$.  The torsion subgroup
scheme $E[n]$ of an elliptic curve comes equipped with a Weil pairing
\[ e_n\colon E[n]\times E[n]\to \mu_n.\] When classifying
$n$-congruences one should take into account what happens to the Weil
pairing. The following proposition is trivial to prove and shows that
$\Aut(\mu_n) \isom (\Z/n\Z)^\times$ classifies the possible types of
$n$-congruences.
\begin{Proposition}
\label{P:n-congruence_mu-auto}
Let $\sigma\colon E[n]\to E'[n]$ be an $n$-congruence.  Then there
exists $\tau_\sigma\in\Aut(\mu_n)$ such that
\[e_n\circ(\sigma\times\sigma)=\tau_\sigma\circ e_n.\] We say that
$\sigma$ is a \emph{direct} $n$-congruence if
$\tau_\sigma(\zeta)=\zeta$, and that $\sigma$ is a \emph{reverse}
$n$-congruence if $\tau_\sigma(\zeta)=\zeta^{-1}$.
\end{Proposition}
In our case, for $n=4$, any $n$-congruence is either direct or
reverse.

Fixing an elliptic curve $E$, we consider the moduli space
$Y_E^+(n)(k)$ of pairs $(E',\sigma)$, where $\sigma\colon E[n]\to
E'[n]$ is a direct $n$-congruence. This moduli space is represented by
a curve $Y_E^+(n)$ over $k$, whose non-singular completion $X_E^+(n)$
is a twist of the modular curve $X(n)$ of full level $n$.  Similarly,
we write $X_E^-(n)$ for the twist of $X(n)$ corresponding to the
moduli space of pairs $(E',\sigma)$ where $\sigma\colon E[n]\to E'[n]$
is a reverse $n$-congruence.

If $\sigma$ is a direct or reverse $n$-congruence, then the
automorphism $\tau_\sigma$ from
Proposition~\ref{P:n-congruence_mu-auto} can be extended to an
automorphism $\Gm\to \Gm$. In this case, we see that $\sigma$ provides
a way of comparing theta groups.

\begin{Definition} 
\label{def:sigma-isom}
Let $\sigma\colon E[n]\to E'[n]$ be a direct or reverse
$n$-congruence.  Given theta groups $\Theta,\Theta'$ for $E[n]$ and
$E'[n]$, we say a morphism $\psi\colon \Theta\to\Theta'$ (as
$\K$-group schemes) is a $\sigma$-isomorphism if it makes the diagram
below commutative.
\[
\xymatrix{
  0\ar[r]&\Gm\ar[r]\ar[d]^{\tau_\sigma}&\Theta\ar[d]^\psi\ar[r]&E[n]\ar[d]^\sigma\ar[r]&0\\
  0\ar[r]&\Gm\ar[r]&\Theta'\ar[r]&E'[n]\ar[r]&0 } \]
\end{Definition}

If $\sigma$ is a direct $n$-congruence, this is the normal notion of
isomorphism for theta groups upon identifying $E[n]$ and $E'[n]$ via
$\sigma$.

\subsection{Shioda modular surfaces}\label{S:shioda_modular_surfaces}

For $n\geq 3$, the moduli space $Y_E^+(n)$ is \emph{fine}, so there is
a universal elliptic curve $E_t$ with a direct $n$-congruence
$\sigma_t\colon E[n]\to E_t[n]$ over $Y_E^+(n)$. The relevant property
for us is that any direct $n$-congruence between $E$ and another
elliptic curve can be obtained by specializing $t$ at the relevant
moduli point on $Y_E^+(n)$. We write $S^+_E(n) \to X_E^+(n)$ for the
minimal fibred surface with generic fibre $E_t$. This is a twist of
Shioda's modular surface of full level $n$.

Given $\xi\in H^1(k,E[n])$, we can twist this surface further. Writing
$k_t$ for the function field of $X_E^+(n)$, we view $\xi$ as an
element of $H^1(k_t,E_t[n])$ and take the $n$-cover $C_{t,\xi}\to E_t$
representing it. Let $S^+_{E,\xi}(n) \to X^+_E(n)$ be the minimal
fibred surface with generic fibre $C_{t,\xi}$.  This is again a twist
of Shioda's modular surface, and is isomorphic to $S^+_E(n)$ over the
splitting field of $\xi$.  In \cite{Mazur-cubics}, the surfaces
$S_E^+(n)$ and $S_{E,\xi}^+(n)$ are called {\em first} and {\em
  second} twists.

We define $S^-_{E,\xi}(n)$ similarly, by using a universal reverse
$n$-congruence $\sigma_t$.

\subsection{Theta modular surfaces}
Alternatively, given an elliptic curve $E$ over $k$ and $\xi\in
H^1(k,E[n])$, we base extend $\Theta_{E,\xi}$ to a theta group over
$k_t$, the function field of $X^+_E(n)$.  Then $\Theta_{E,\xi}\times_k
k_t=\Theta_{E_t,\xi_t}$ for some $\xi_t\in H^1(k_t,E_t[n])$. Just as
in Section~\ref{S:shioda_modular_surfaces}, we take the $n$-cover
$C_{t,\xi_t}\to E_t$ representing $\xi_t$. Then we define the {\em
  theta modular surface} to be the minimal fibred surface
$T_{E,\xi}^+(n) \to X_E^+(n)$ with generic fibre $C_{t,\xi_t}$.

By construction, the fibers of $T_{E,\xi}^+(n) \to X_{E}^+(n)$ are
$n$-covers with a prescribed theta group.  Since $\ob_{E,n}(\xi)$ is a
function of $\Theta_{E,\xi}$, we see that $\ob_{E_t,n}(\xi_t)$ is the
base change of $\ob_{E,n}(\xi)$ to $k_t$. In particular, if
$\ob_{E,n}(\xi)=0$ then $C_{t,\xi_t}$ admits a degree $n$ model in
$\PP^{n-1}$, with a linear action of $E_t[n]$. In that case, it
follows that $T_{E,\xi}^+(n)$ is birational to a surface in
$\PP^{n-1}$, with an action of $E[n]$ through $\Theta_{E,\xi}$.  This
allows us to use invariant theory to write down models of
$T_{E,\xi}^+(n)$.

Interestingly enough, $\nu(t)=\xi_t-\xi$ is not necessarily
trivial. In fact, Theorem~\ref{T:theta_shift} proves that $\nu(t)$ is
$2$-torsion (as is seen by applying the result to the elliptic curves
$E\times_k k_t$ and $E_t$).  In particular, if $n$ is odd then $\nu(t)
= 0$ and the surfaces $S^+_{E,\xi}(n)$ and $T^+_{E,\xi}(n)$ are
isomorphic.

We define $T^-_{E,\xi}(n)$ similarly, by reversing the order of
multiplication on $\Theta_{E,\xi}$.  
Since $(\Z/4\Z)^\times = \{\pm 1\}$, this is 
sufficient to define $T_{E,\xi}(4)$.

\section{Computing $4$-covers}
\label{sec:4cover}

Let $E$ be an elliptic curve and let $\xi\in H^1(k,E[4])$. In this
section, we write $D_{E,\xi}$ for the corresponding $4$-cover of
$E$. We have $2\xi\in H^1(k,E[2])$ and write $C_{E,2\xi}$ for the
corresponding $2$-cover of $E$. Note that the $4$-cover $D_{E,\xi}\to
E$ naturally factors as $D_{E,\xi}\to C_{E,2\xi}\to E$. It turns out
to be advantageous to study $4$-covers via this intermediate
structure.

\begin{Definition} Let $C\to E$ be a $2$-cover. A $2$-cover of $C$ is
  a cover $D\to C$ such that the composition of covers $D\to C\to E$
  is a $4$-cover. If we want to emphasize that $D$ is a $2$-cover of a
  $2$-cover, and not of an elliptic curve directly, we say that $D\to
  C$ is a \emph{second} $2$-cover.
\end{Definition}

In this section we are concerned with finding models of second
$2$-covers. We fix a $2$-cover $C\to E$ over a field $k$ with $\Char k
\not=2,3$.  We assume $\ob_{E,2}(C)=0$, so that $C$ has a model
\begin{equation}
\label{BQ}
C\colon \quad Y^2=G(X,Z),
\end{equation}
where $G$ is a binary quartic form, say
\[ G(X,Z) = a X^4 + b X^3 Z + c X^2 Z^2 + d X Z^3 + e Z^4. \] The
classical invariants of $G$ are
\begin{equation}
\label{invIJ}
\begin{aligned}
I &= 12 a e - 3 b d + c^2, \\
J &= 72 a c e - 27 a d^2 - 27 b^2 e + 9 b c d - 2 c^3.
\end{aligned}
\end{equation}
As observed by Weil, a model for $E$ is given by
\begin{equation}
\label{JacBQ}
E\colon\quad y^2 = x^3 + \wA x + \wB \text{ where } 
\wA = -I/48 \text{ and }\wB =-J/1728.
\end{equation}
In the next two sections we assume for simplicity that $ae \not= 0$.

\subsection{Models of $2$- and $4$-covers with trivial obstruction}
\label{S:24trivial}
In this section we review classical $4$-descent, as described in
\cites{ANTS:4desc,MSS,Stamminger,Womack} and
implemented in Magma \cite{magma}.

If $D\to C$ is a second $2$-cover such that $\ob_4(D)=0$ then $D$
admits a degree $4$ model in $\PP^3$, say
\[D\colon \quad Q_1=Q_2=0,\] where $Q_1,Q_2\in k[\x]=k[x_1,x_2,x_3,x_4]$ are
quadratic forms. Conversely, any such smooth quadric intersection is a
$4$-cover of an elliptic curve.

Let $A_i$ be the symmetric matrix such that
$Q_i(\x)=\frac{1}{2}\x^TA_i\x$.  Then the intermediate $2$-cover has a
model~\eqref{BQ} in weighted projective space, with
$G(X,Z)=\det(XA_1+ZA_2)$. In particular $\ob_2(C)=0$.  Let $T_1,T_2
\in k[\x]$ be the quadratic forms given by
$T_i(\x)=\frac{1}{2}\x^TB_i\x$ where
\[ \adj((\adj A_1)X+(\adj A_2)Z) 
 =  a^2 A_1 X^3 + a B_1 X^2 Z - e B_2 X Z^2 + e^2 A_2 Z^3. \]
Then the covering map $D \to C$ is given by $(X:Z:Y) = (T_1:T_2:J)$ where
\[ J = (1/4)
\frac{\partial(Q_1,Q_2,T_1,T_2)}{\partial(x_1,x_2,x_3,x_4)}. \]

Let $F=k[\theta]$ be the \'etale algebra over $k$ generated by a root
$\theta$ of $g(x)=G(x,1)$. A generic calculation shows that the
quadratic form
\begin{equation}
\label{rank1}
 \Xi = \theta^{-1} e Q_1 + T_1 - \theta T_2 + \theta^2 a Q_2 
\end{equation}
in $F[\x]$ has rank $1$ and satisfies $a N_{F/k} (\Xi) = J^2$.
Specialising at any sufficiently general $\x \in k^4$ shows that there
exist $\alpha\in F^\times$ and $r\in k^\times$ with
$N_{F/k}(\alpha)=ar^2$.  Moreover the class of $\alpha$ in
$F^\times/F^{\times2}\K^\times$ only depends on the isomorphism class
of the $2$-cover $D \to C$.

Conversely, given a $2$-cover $C$ of the form~\eqref{BQ}, we can
construct its $2$-covers $D$ with $\ob_4(D)=0$ in the following way.
Let $F = k[\theta]$ as above, and suppose that $\alpha\in F^\times$
and $r\in k^\times$ satisfy $N_{F/k}(\alpha)=ar^2$.  We consider the
equation
\begin{equation}
\label{oldstart}
\alpha(X - \theta Z) = (x_1 + x_2 \theta + x_3 \theta^2 + x_4
\theta^3)^2.
\end{equation}
Expanding in powers of $\theta$ gives $4$ equations of the form
\[ \text{(linear form in $X$ and $Z$)} = \text{(quadratic form in
  $x_1, \ldots, x_4$)}. \] Taking the norm $N_{F/\K}$ and then
extracting a square root gives, upon choosing the sign of $r$, a further
equation
\[ r Y = N_{F/\K}( x_1 + x_2 \theta + x_3 \theta^2 + x_4 \theta^3). \]
Taking linear combinations of these equations gives expressions for
$X,Y,Z$ in terms of $x_1, \ldots ,x_4$, and two further quadratic
equations in $x_1,\ldots,x_4$ only. These define a $2$-cover $D_\alpha
\to C$. Moreover the isomorphism class of this $2$-cover only depends
on the class of $\alpha$ in $F^\times/F^{\times2}\K^\times$.

The two constructions just presented are inverse to one another.  We
thus obtain the following proposition. In stating it we use our
freedom to multiply $G(X,Z)$ by a square to reduce to the case $r=1$.
\begin{Proposition}
\label{P:4cov_trivob}
Let $C \to E$ be a $2$-cover. If there is a second $2$-cover $D \to C$
with $\ob_4(D)=0$ then $C$ has a model of the form
$Y^2=N_{F/k}(\alpha(X-\theta Z))$. Moreover if $C$ takes this form
then the collection of all $2$-covers $D$ of $C$ with $\ob_4(D)=0$ is
given by
\begin{equation}
\label{subq}
\ker(N_{F/\K}\colon F^\times/F^{\times2}\K^\times\to \K^\times/\K^{\times 2})
\end{equation}
via the map $\delta\mapsto D_{\alpha \delta}$.
\end{Proposition}

\begin{Remark}
\label{4desc-rems}
(i) Strictly speaking we should specify a choice of square root of
$N_{F/\K}(\delta)$, otherwise the $2$-covers $D_{E,\xi}$ and
$D_{E,-\xi}$ of $C_{E,2\xi}$, differing by the automorphism $Y \mapsto
-Y$ of $C_{E,2\xi}$,
cannot be distinguished. \\
(ii) Let $g'(x)$ be the derivative of $g(x) =G(x,1)$.  It is sometimes
convenient to write the equations for $D_\alpha$ as
\begin{equation}
\label{traceq}
\tr_{F/k} \left( \frac{x^2 }{\alpha  g'(\theta)} \right)
= \tr_{F/k}  \left( \frac{  \theta x^2}{\alpha g'(\theta)} \right) 
= 0,
\end{equation}
where $x=x_1+x_2\theta+x_3\theta^2+x_4\theta^3$. \\
(iii) The group~\eqref{subq} may be identified with a certain subgroup
of $H^1(k,E[2])/\langle 2\xi \rangle$ where $2\xi$ is the class of the
$2$-cover $C \to E$.  See \cite{ANTS:4desc} for further details.
\end{Remark}

\subsection{Models for twists of second $2$-covers}
\label{S:Dprime_nu}
Let $D \to C$ be a second $2$-cover with $\ob_4(D)=0$.  By
Proposition~\ref{P:4cov_trivob} we may assume that $C$ has a model
\begin{equation}
\label{BQ1}
C: \quad Y^2=G(X,Z)=N_{F/\K}(\alpha(X-\theta Z)),
\end{equation}
and that $D = D_\alpha$.  In this section we are interested in models
for \emph{any} $2$-cover of $C$; not just the ones with trivial
obstruction (as a $4$-cover). These covers are parametrized by
$H^1(k,E[2])$.

We recall that $C$ is a $2$-cover of the elliptic curve $E\colon
y^2=f(x)= x^3+\wA x+\wB$ given by~\eqref{JacBQ}.  Let $L=\K[\varphi]$
be the \'etale algebra over $k$ generated by a root $\varphi$ of
$f(x)$.  It is well known that
\begin{equation}\label{H1isom}
  H^1(k,E[2])\isom \ker(N_{L/k}\colon L^\times/ L^{\times2}\to k^\times/k^{\times2}).
\end{equation}
In fact this is the special case of Remark~\ref{4desc-rems}(iii) with
$C = E$.

The algebra $L$ is related to $F$ in the following way. We have that
$g(x)=(x-\theta)h(x)$, for some cubic $h(x)\in F[x]$. Then
$F[x]/(h(x))=L\otimes_\K F$, which we denote by $LF$. As an algebra
over $\K$, it is obtained by formally adjoining two roots, say
$\theta$ and $\tilde{\theta}$, of $g(x)$. Let $\sigma\in \Aut_\K(LF)$
be the involution that swaps $\theta$ and $\tilde{\theta}$.  We write
$M$ for the subalgebra of $LF$ fixed by $\sigma$. This is the \'etale
algebra of unordered pairs of roots of $g(x)$ and it has degree
$[M:\K]=6$.  We may identify $L$ as a subalgebra of $M$ via
\begin{equation}
\label{embphi}
\varphi = -(a \theta \tilde{\theta} - c/3 + e/(\theta \tilde{\theta}))/4. 
\end{equation}

We fix a basis $m_1, \ldots, m_6$ for $M$ over $\K$, and put
$\tilde{\alpha} = \sigma(\alpha)$.  Let $\nu \in L^\times$ with
$N_{L/\K}(\nu) = s^2$ for some $s \in \K^\times$. We show how to
construct a twist $\D_\nu\to C$ of $D \to C$.  This will turn out to
be the twist by the element of $H^1(k,E[2])$ corresponding to $\nu$
under the isomorphism~\eqref{H1isom}.

We consider the equation
\begin{equation}
\label{start}
N_{LF/M}(\alpha(X-\theta Z))=\alpha 
\tilde{\alpha}(X - \theta Z)(X - \tilde{\theta} Z) 
= \nu (y_1 m_1 + \ldots + y_6 m_6)^2. 
\end{equation}
Expanding and taking coefficients of $m_1, \ldots, m_6$ gives $6$
equations of the form
\[ \text{(quadratic form in $X$ and $Z$)} = \text{(quadratic form in
  $y_1, \ldots, y_6$)}. \] Taking the norm $N_{M/L}$ and then
extracting a square root gives
\begin{equation}
\label{eqn:Y}
    Y = \nu \, N_{M/L}(  y_1 m_1 + \ldots + y_6 m_6 ) 
\end{equation}
and hence $3$ equations of the form
\[ \text{(linear form in $Y$)} = \text{(quadratic form in $y_1,
  \ldots, y_6$)}. \] Taking linear combinations to eliminate $X^2, XZ,
Z^2$ and $Y$ leaves $5$ quadratic forms in $y_1, \ldots, y_6$.  These
define $\D_\nu \subset \PP^5$, a genus $1$ curve of degree $8$.  In
fact $\D_\nu$ is a $2$-cover of $C$. Equations for the covering map
$\D_\nu \to C$ may be computed as follows.  Again starting
from~\eqref{start}, we take the norm $N_{LF/F}$ and then extract a
square root to give
\[ \pm \alpha (X - \theta Z)  Y = s \, N_{LF/F}( y_1 m_1 + \ldots
+y_6 m_6 ).\] The sign choice here may be absorbed into the cubic
norm.  Using~(\ref{eqn:Y}) to eliminate $Y$ and then cancelling the
common factor $y_1 m_1 + \ldots + y_6 m_6$ gives $12$ equations of the
form
\[ \text{(bilinear form in $X,Z$ and $y_1, \ldots, y_6$)} =
\text{(quadratic form in $y_1, \ldots, y_6$)}. \] These equations
together with~(\ref{eqn:Y}) define the covering map $\D_\nu \to C$.

\begin{Remark} 
\label{rems1}
(i) If we just eliminate $Y$ from the $6+3+12$ equations listed above,
then we get $20$ quadrics in $X,Z,y_1, \ldots, y_6$. These define a
genus $1$ curve embedded in $\PP^7$ via a complete linear system of
degree~$8$.  However we will see that working with $\D_\nu \subset
\PP^5$ has some advantages. \\
(ii) Taking $\nu=1$ gives a $2$-cover $\D_1 \to E$ that is isomorphic
to $D \to C$. Indeed on comparing~\eqref{oldstart} and~\eqref{start}
we see that an isomorphism is given by
\[y_1m_1+\cdots+y_6m_6 =(x_1+x_2\theta+x_3\theta^2+x_4\theta^3)
(x_1+x_2\tilde{\theta}+x_3\tilde{\theta}^2+x_4\tilde{\theta}^3).\] In
fact the $y_i$ span the same space as the $2\times 2$ minors of the
$2\times 4$ matrix of partial derivatives
of the quadratic forms defining $D$. \\
(iii) A generic calculation shows that $\D_\nu \subset \PP^5$ has
degree $8$ and its homogeneous ideal is (minimally) generated by $5$
quadrics and $2$ cubics.  However the $5$ quadrics
are sufficient to define the curve set-theoretically. \\
(iv) Let $z \in L^\times$ correspond under the
isomorphism~\eqref{H1isom} to the class of $C \to E$. By
\cite{CremonaBQ}*{Equation (3.1)} we have $z \in M^{\times2}$, and so
$z$ is a Kummer generator for the quadratic extension $M/L$.
Absorbing $z$ into the squared factor on the right of~\eqref{start} we
see that $\D_\nu$ and $\D_{\nu z}$ are isomorphic as curves. However
as $2$-covers of $C$ they differ by the automorphism $Y \mapsto -Y$.
\end{Remark}

We prove an analogue of Proposition~\ref{P:4cov_trivob}.
\begin{Proposition}
\label{P:4cov}
Let $C$ be the $2$-cover~\eqref{BQ1}.  Then the collection of all
$2$-covers of $C$ is given by $\ker(N_{L/\K}\colon
L^\times/L^{\times2}\to \K^\times/\K^{\times 2})$ via the map $\nu
\mapsto \D_\nu$.
\end{Proposition}

\begin{proof}
  Let $\eta \in H^1(k,E[2])$ map to the class of $\nu \in L^\times$
  under the isomorphism~\eqref{H1isom}. To prove the proposition, we
  show that $\D_\nu \to C$ is the twist of $\D_1 \to C$ by $\eta$.

  Let $\Lbar = L \otimes_k \Kbar$. Note that $L$ is the coordinate
  ring of $E[2]\setminus\{0_E\}$, so $\Lbar = \Map(E[2](\Kbar)
  \setminus \{0_E\},\Kbar)$.  There is a homomorphism of Galois
  modules
\begin{align*}  
 w : E[2] & \to \mu_2(\Lbar) \\
S & \mapsto (T \mapsto e_2(S, T)). 
\end{align*}
Noting that $M$ is an $L$-algebra and a $6$-dimensional $\K$-vector
space, we see that $L^\times$ acts linearly on $\PP(M) = \PP^5$.  In
particular $S\in E[2]$ acts on $\PP^5$ via $w(S)$ and this action
restricts to $\D_1$.

If $g(x)$ has roots $\theta_1, \ldots, \theta_4$ then
$\Kbar(\D_1)=\Kbar(C)(\sqrt{f_{12}},\sqrt{f_{13}})$ where $f_{ij} =
(X- \theta_iZ)(X-\theta_jZ)/Z^2$. Since $\D_1$ is a homogeneous space
under $E$, there is an action of $E[2]$ on $\D_1$. This is given by
$\sqrt{f_{12}} \to \pm \sqrt{f_{12}}$ and $\sqrt{f_{13}} \to \pm
\sqrt{f_{13}}$. It follows from the definition of the Weil pairing
that this action agrees with the one defined in the last paragraph.

Let $\eta$ be represented by the cocycle $\sigma \mapsto \eta_\sigma$.
Then we have $w(\eta_\sigma) = \sigma(\gamma) \gamma^{-1}$ and $\nu =
\gamma^2$ for some $\gamma \in \Lbar$.  There is a commutative diagram
\[\xymatrix{  \D_\nu \ar[r]^{\cdot \gamma} \ar[d] & \D_1 \ar[d] \\
C \ar@{=}[r] & C } \]
Therefore $\D_\nu \to C$ is the twist of $\D_1 \to C$ via 
the cocycle $\sigma \mapsto \sigma(\gamma) \gamma^{-1}$, 
and this completes the proof.
\end{proof}

\section{Theta groups and the shift}
\label{sec:thetashift}

In this section we prove the following theorem. We work over a field
$\K$ of characteristic not dividing $n$.

\begin{Theorem}
	\label{T:theta_shift}
	Let $E,E'$ be elliptic curves over $\K$ with a direct or a
        reverse $n$-congruence $\sigma\colon E'[n]\to E[n]$. Then
        there exists $\nu\in H^1(\K,E[n])$, depending only on
        $E,E',\sigma$, such that for any $\xi\in H^1(\K,E[n])$ and
        $\xi'\in H^1(\K,E'[n])$ we have
	\begin{enumerate}
		\item $2\nu=0$ (in particular, if $n$ is odd then $\nu=0$).
		\item $\xi=\sigma_*(\xi')+\nu$ if and only if $\Theta_{E,\xi}$ and $\Theta_{E',\xi'}$ are $\sigma$-isomorphic.
	\end{enumerate}
\end{Theorem}

\begin{proof}
  We start by comparing the trivial theta groups
  $\Theta_{E},\Theta_{E'}\subset\GL_n$.  There is an isomorphism
  $\psi$ defined over $\Kbar$ making the following diagram commute
  \begin{equation}
\label{ThetaEdiag2}
\begin{aligned}
  \xymatrix{ 0 \ar[r] & \Gm \ar[r] \ar[d]^{\tau_\sigma} & \Theta_{E'}
    \ar[r] \ar[d]^{\psi} & E'[n]
    \ar[r] \ar[d]^{\sigma} & 0  \\
    0 \ar[r] & \Gm \ar[r] & \Theta_{E} \ar[r] & E[n] \ar[r] & 0
    \rlap{.}  }
\end{aligned}
\end{equation}
Then $\Theta_{E'}$ is the twist of $\Theta_E$ by the cocycle $\rho
\mapsto \rho(\psi) \psi^{-1}$. This cocycle takes values in
$\Aut(\Theta_E) \isom \Hom(E[n],\Gm) \isom E[n]$, and so gives a class
$\nu \in H^1(\K,E[n])$.  If $n$ is odd then $\nu = 0$ by
\cite{descI}*{Lemma 3.11}.  We now adapt the argument to the case
where $n$ is even.

The automorphism $[-1]$ of $E$ lifts to $\iota \in \GL_n(\K)$.  For
each $T \in E[n](\Kbar)$ we pick a matrix $M_T \in \Theta_E(\Kbar)$,
representing translation by $T$, such that
\begin{equation}
\label{iota}
 \iota M_T \iota^{-1} = M_T^{-1}. 
\end{equation}
This condition determines the scaling of $M_T$ up to a choice of sign.
If $M_{S+T} = \lambda M_S M_T$ then conjugating by $\iota$ shows that
$\lambda^2 = e_n(S,T)$ and so
\begin{equation}
\label{howmult}
  M_{S+T} = \pm e_n(S,T)^{1/2} M_S M_T \quad 
\text{ for all } S,T \in E[n](\Kbar).
\end{equation}

We claim that $M_T^n = 1$. Indeed for a suitable choice of
co-ordinates $x_0, \ldots, x_{n-1}$, defined over $\Kbar$, we have
$M_T: x_i \mapsto \alpha \zeta^i x_i$ and $\iota : x_i \mapsto \beta
x_{n-i}$ for some $\alpha,\beta \in \Kbar$ and $\zeta \in
\mu_n$. By~(\ref{iota}) we have $\alpha^2=1$. Since $n$ is even it
follows that $M_T^n=1$. This proves the claim.

In exactly the same manner we pick $M'_T \in \Theta_{E'}(\Kbar)$ for
all $T \in E'[n](\Kbar)$. We can then choose $\psi$ so that
$\psi(M'_T) = \pm M_{\sigma(T)}$ for all $T \in E'[n](\Kbar)$. Indeed
if we make this true on a basis for $E'[n]$, then the rest follows
by~(\ref{howmult}) and its analogue for $E'$.

Let $\rho \in \Gal(\Kbar/\K)$.  Since $\iota$ is defined over $\K$, it
follows from~(\ref{iota}) that
\[ \rho(M_T) = \pm M_{\rho(T)} \quad \text{ for all } T \in
E[n](\Kbar). \] Likewise
\[ \rho(M'_T) = \pm M'_{\rho(T)} \quad \text{ for all } T \in
E'[n](\Kbar). \] The cocycle $\rho \mapsto \rho(\psi) \psi^{-1}$ now
takes values in $\Hom(E[n],\mu_2) = E[2]$. Therefore $\nu$ is in the
image of the natural map $H^1(\K,E[2]) \to H^1(\K,E[n])$, and this
proves (i).

It also follows that $\Theta_{E'}$ is $\sigma$-isomorphic to
$\Theta_{E,\nu}$. For the general statement, we choose a
$\Kbar$-isomorphism $\psi'\colon \Theta_{E',\xi'}\to \Theta_{E'}$.
Then the cocycle $\rho\mapsto\rho(\psi\psi')(\psi\psi')^{-1}$
represents the class $\sigma_*(\xi')+\nu$. It follows that
$\Theta_{E',\xi'}$ is $\sigma$-isomorphic to
$\Theta_{E,\sigma_*(\xi')+\nu}$.  This proves (ii).
\end{proof}

\section{Geometry of the Shioda and theta modular surfaces}
\label{sec:shioda}

In this section we give two particular models of the Shioda and theta
modular surfaces of level $4$.  Any other Shioda or theta modular
surface of level $4$ will be a twist of one of these, so these
particular models are convenient for studying the geometry of these
surfaces.

\subsection{The universal elliptic curve of level $4$}
\label{S:E4}

The Legendre form
\[y^2=x(x-1)(x-\lambda)\] provides a family of elliptic curves
$E_\lambda$ with a prescribed isomorphism $(\Z/2\Z\times\mu_2)\to
E_\lambda[2]$. The parameter $\lambda$ gives an isomorphism
$Y(2)\simeq \PP^1\setminus \{0,1,\infty\}$.  Setting $\lambda=
(1-t^2)^2/(1+t^2)^2$ we obtain a family of elliptic curves $E_t$ with
a prescribed isomorphism $(\Z/4\Z\times\mu_4)\to E_t[4]$. The parameter
$t$ gives an isomorphism $Y(4)\simeq \PP^1\setminus \{0,\infty,\pm 1,
\pm i\}$ and the expression for $\lambda$ in terms of $t$ is an
explicit realisation of the map $Y(4)\to Y(2)$.

\subsection{The theta modular surface of level $4$}
\label{S:T4-geom}
The quadric intersection
\begin{equation}
\label{hessefamily}
 D_t\colon \quad \left\{
\begin{aligned}
t(x_0^2 + x_2^2) + 2 x_1 x_3 &= 0 \\
t(x_1^2 + x_3^2) + 2 x_0 x_2 &= 0
\end{aligned} \right\} \subset \PP^3 
\end{equation}
is a $4$-cover of $E_t$. By varying $t$ these curves give a genus $1$
fibration on the surface
\begin{equation}\label{E:quartic_surface}
x_0x_2(x_0^2+x_2^2)-x_1x_3(x_1^2+x_3^2)=0.
\end{equation}

The action of $E_t[4]$ on $D_t$ is generated by the transformations
$x_{\nu} \mapsto x_{\nu+1}$ and $x_{\nu} \mapsto i^\nu x_\nu$, where
we read the subscripts mod $4$.  (One convenient way to check an
automorphism of a genus $1$ curve is geometrically a translation map,
is to check it has no fixed points.) The quadric intersections $D_t$
therefore all have the same theta group, and
so~\eqref{E:quartic_surface} is a model for the theta modular surface
$T(4)$.

The surface $T(4)$ is isomorphic to the Fermat quartic
$\{u_0^4-u_1^4+u_2^4-u_3^4=0\} \subset \PP^3$ via the change of
coordinates
\[(u_0:u_1:u_2:u_3)  =(x_0+x_2:x_0-x_2:x_1-x_3:x_1+x_3).\]

\subsection{Shioda's modular surface of level $4$}
\label{sec:geomS4}

The relative elliptic curve $E_t/Y(4)$ provides us with an open part
of the Shioda modular surface. In order to find a suitable completion,
we construct $E_t$ as a (trivial) $4$-cover of itself. The
intermediate $2$-cover is
\begin{equation}\label{eqn:E}
  C_t\colon \quad Y^2 = XZ(X-Z)(X-\la Z) \quad \text{ with }\lambda=\frac{(1-t^2)^2}{(1+t^2)^2}.
\end{equation}
A small adaptation (needed since $a=e=0$) of the construction in
Section~\ref{S:Dprime_nu} leads to the second $2$-cover given by
\begin{equation}
\label{sixeqns}
\begin{aligned}
(X-Z)(X-\la Z) & =  y_1^2 &&\qquad& (1-\la) X Z & =  y_2^2 \\
X(X-\la Z) & = y_3^2 &&& \la (X-Z) Z & =  y_4^2 \\
(X-\la Z)Z & = y_5^2 &&& X (X-Z) & =  y_6^2 
\end{aligned}
\end{equation}
and 
\[ Y = \frac{1+t^2}{2t} y_1 y_2 = \frac{1+t^2}{1-t^2} y_3 y_4 = y_5
y_6. \] Taking linear combinations to eliminate $X^2, XZ, Z^2$ and $Y$
gives $5$ quadrics that define a degree $8$ curve in $\PP^5$:
\begin{equation}
\label{fivequads}
\D_t\colon\quad \left\{
\begin{aligned}
    y_1^2 + y_4^2 - y_6^2 & = 0 \\
    y_2^2 - y_3^2 + y_6^2 & = 0 \\
    y_2^2 + y_4^2 - y_5^2 & = 0 \\
    (1 - t^2) y_1 y_2 - 2 t y_3 y_4& = 0 \\
    (1 + t^2) y_1 y_2 - 2 t y_5 y_6& = 0 
\end{aligned}
\right\} \subset \PP^5
\end{equation}
The first $3$ quadrics in~\eqref{fivequads} define a surface $S_0
\subset \PP^5$ of degree $8$ with exactly $16$ ordinary double points
as singularities ($8$ defined over $\Q$ and the remaining $8$ over
$\Q(i)$). Each curve $\D_t$ passes through all $16$ singular
points. Furthermore, these are the points with $Y=0$, and so make up
the fibre of $\D_t\to E_t$ above $0_{E_t}$. In particular, $\D_t$ has
a rational point over $0_{E_t}$, so $\D_t \to E_t$ is indeed the
trivial $4$-cover.  Blowing up the singular points gives
Shioda's modular surface $S(4)$.

The action of $E_t[2] \isom \Z/2\Z \times \mu_2$ on $C_t$ is generated
by
\begin{align*}
(X:Z:Y) & \mapsto (X - \lambda Z : X - Z : (\lambda - 1)Y), \\
(X:Z:Y) & \mapsto (\lambda Z : X : -\lambda Y),
\end{align*}
and the action of $E_t[4] \isom \Z/4\Z \times \mu_4$ on $\D_t$ is generated by 
\begin{equation}
\label{act4}
\begin{aligned}
(y_1:y_2:y_3:y_4:y_5:y_6)&\mapsto(-y_2:y_1:-y_3:y_4:-y_6:y_5), \\
(y_1:y_2:y_3:y_4:y_5:y_6)&\mapsto(iy_1:-iy_2:y_4:y_3:y_6:y_5).
\end{aligned}
\end{equation}

The subgroup of $\Pic S(4)$ invariant under the action~\eqref{act4}
was computed in \cite{BH}, and shown to be free of rank $2$ generated
by divisor classes $I$ and $F$, where $F$ is the class of a fibre, and
$2I$ is linearly equivalent to the sum of the $16$ sections (i.e. the
blow-ups of the $16$ singular points on $S_0$).  Our surface $S_0$ is
the image of $S(4)$ under the morphism to $\PP^5$ given by the
complete linear system $|I+F|$.

More generally there is a natural action of $\G =
\ASL(\Z/4\Z\times\mu_4)$ on $S(4)$.  The corresponding automorphisms
of $S_0$ are again given by changes of co-ordinates on $\PP^5$. The
surfaces $S^{\pm}_{E,\xi}(4)$ are twists of $S(4)$ by cocycles taking
values in $\G$, and so each must admit a model in a $5$-dimensional
Brauer-Severi variety. Our calculations in
Sections~\ref{S:twist_shioda} and~\ref{S:rev} show that if
$\ob_{E,4}(\xi) = 0$ then this Brauer-Severi variety is
trivial. Indeed we show how to
 write down a model for
$S^{\pm}_{E,\xi}(4)$ as a complete intersection of quadrics in
$\PP^5$.

\begin{Remark}
\label{get-fib}
Our calculations also give the genus $1$ fibration, but in fact this
may be recovered directly from the equations for the surface.  Indeed,
if we take a complement to the $3$-dimensional space of quadrics
vanishing on the surface, inside the $6$-dimensional space of quadrics
vanishing at the singular points, then this defines a map to $\PP^2$
with image a conic. For example, with $S_0$ as above, the map is given
by $(X_1:X_2:X_3) = (y_1y_2:y_3y_4:y_5y_6)$ and the conic is $X_1^2 +
X_2^2 = X_3^2$.  Parametrising this conic gives the required map to
$\PP^1$.
\end{Remark}

\section{Computing twists of $S(4)$ and $T(4)$ in the direct case}
\label{S:twist_shioda}

In this section we take $\xi \in H^1(k,E[4])$ with $\ob_{E,4}(\xi)=0$
and compute models for $S^+_{E,\xi}(4)$ and $T^+_{E,\xi}(4)$.

\subsection{The twisted universal elliptic curve of level $4$}
Let $E/\K$ be the elliptic curve $y^2 = f(x) = x^3 + \wA x + \wB$.  If
$P = (x_P,y_P)$ is a point on $E$ then the $x$-coordinate of $2P$ is
$\gamma(x_P)$ where
\[ \gamma(x) = (x^4 - 2 \wA x^2 - 8 \wB x + \wA^2)/(4 (x^3 + \wA x +
\wB)). \] The elliptic curves directly $n$-congruent to $E$ are
parametrised by the non-cuspidal points of $X^+_E(n)$. This is a twist
of the usual modular curve $X(n)$.

\begin{Lemma}
\label{lem:4congr}
The elliptic curves directly $4$-congruent to $E$ are
\[ E_t : \quad y^2 = - f(t) (x - \gamma(t)) (x^3 + \wA x + \wB) \]
with base point $(x,y) = (\gamma(t),0)$, where $t$ is a co-ordinate on
$X_E^+(4) \isom \PP^1$, and the original curve $E$ corresponds to $t =
\infty$.  Moreover the cusps of $X_E^+(4)$
are the roots of $d(t)=0$ where
\[ d(x) = x^6 + 5 \wA x^4 + 20 \wB x^3 - 5 \wA^2 x^2 - 4 \wA \wB x -
\wA^3 - 8 \wB^2 \] is the $4$-division polynomial of $E$ (divided by
$2$).
\end{Lemma}
\begin{proof}
  Putting $E_t$ in Weierstrass form, and make a change of co-ordinates
  on $X^+_E(4)$, gives the family of curves computed by Silverberg.
  Indeed our parameter $t$ on $X^+_E(4)$ and the parameter $t$, here
  denoted by $t_\text{Silverberg}$, in \cite{Silverberg}*{Theorem 4.1}
  are related by
  \[ t = \frac{-(4 \wA^3 + 27 \wB^2)}{18 \wA \wB
    t_{\text{Silverberg}}} - \frac{3 \wB}{2 \wA}. \] An alternative
  proof, also treating the cases $j(E)=0,1728$, and leading to the
  lemma as stated here, is given in \cite{BD}*{Proposition 7.2}.  See
  also Remark~\ref{xe4viainvthy}(i).
\end{proof}

\subsection{The twisted theta modular surface of level $4$}
Let $D = \{ Q_1 = Q_2 = 0 \} \subset \PP^3$ be a quadric intersection
with Jacobian $E$. Then $D=D_{E,\xi}$ for some $\xi\in H^1(\K,E[4])$
with $\ob_{E,4}(\xi)=0$. We use invariant theory to compute a model
for $T^+_{E,\xi}(4)$ as a quartic surface in $\PP^3$, together with
its genus $1$ fibration over $X^+_E(4)$.

We identify the quadratic forms $Q_1$ and $Q_2$ with $4 \times 4$
symmetric matrices $A_1$ and $A_2$ via $Q_i(x_1, \ldots, x_4) =
\tfrac{1}{2} \x^T A_i \x$.  We then define $G(X,Z)$ by
\begin{equation}
\label{getint}
G(X,Z) = \det(XA_1+ZA_2)  =  a X^4 + b X^3 Z + c X^2 Z^2 + d X Z^3+e Z^4,
\end{equation} 
so that $D$ is a $2$-cover of $C: Y^2 = G(X,Z)$. As in
Section~\ref{sec:4cover}, we assume $ae \not=0$ and let $F =
\K[\theta]$ where $\theta$ is a root of $g(x) = G(x,1)$.

Let $T_1,T_2 \in k[\x]$ be the quadratic forms defined in
Section~\ref{S:24trivial}.  The Hessian, as defined in
\cite{g1hessians}, is an $\SL_2 \times \SL_4$-equivariant map from the
space of quadric intersections to itself. It is given by
\[ (Q_1, Q_2) \mapsto (Q'_1,Q'_2) = ( 6 T_2 - c Q_1 - 3 b Q_2 , 6 T_1
- c Q_2 - 3 d Q_1 ). \]

\begin{Lemma}
\label{prop:int2}
The quadric intersection
\begin{equation}
\label{Ds}
 D_t  \colon \quad  \{ 12 t Q_1 + Q'_1 = 12 t Q_2 + Q'_2 = 0 \} \subset \PP^3
\end{equation}
has intermediate $2$-cover (after cancelling a factor $12^4$)
\[ C_t \colon \quad
 Y^2 = G_t(X,Z) = a N_{F/\K} \big((t+\mu)X-(\theta t + \lambda)Z\big), \]
where 
\begin{equation}
\label{lamu}
\begin{aligned}
\lambda &=  (6 g'(\theta) - \theta g''(\theta))/24 
  = -(c \theta^2 + 3 d \theta + 6 e)/(12\theta), \\
\mu &=  - g''(\theta)/24 = -(6 a \theta^2 + 3 b \theta + c)/12.
\end{aligned}
\end{equation}
\end{Lemma}
\begin{proof}
  In principle this may be checked by a generic calculation.  To make
  the calculation practical we consider the case $Q_1 = \sum_{i=1}^4
  \xi_i x_i^2$ and $Q_2 = -\sum_{i=1}^4 \xi_i \theta_i x_i^2$.  Then
  $g(X) = 2^4 (\prod_{i=1}^4 \xi_i) \prod_{i=1}^4(X - \theta_i)$, and
  $(Q_1,Q_2)$ has Hessian $(Q'_1,Q'_2)$ where $Q'_1 = 12\sum_{i=1}^4
  \xi_i \mu_i x_i^2$ and $Q'_2 = -12\sum_{i=1}^4 \xi_i \lambda_i
  x_i^2$.  Computing the intermediate $2$-cover, by the method used
  in~\eqref{getint}, gives the equation for $C_t$ as stated.
\end{proof}

\begin{Corollary}
\label{C:theta_surface_equation}
Let $\xi\in H^1(k,E[4])$ with $D_{E,\xi} =
\left\{Q_1=Q_2=0\right\}\subset\PP^3$.  Then the surface
$T^+_{E,\xi}(4)$ has a model
\[\{Q_1Q'_2-Q_2Q'_1=0\}\subset\PP^3\] with genus $1$ fibration $D_t$
as given in Lemma~\ref{prop:int2}.
\end{Corollary}

Let $I$ and $J$ be the invariants~(\ref{invIJ}) of the binary quartic
$G(X,Z)$. Then $E$ has Weierstrass equation $y^2 = x^3 + \wA x + \wB$
where $\wA = -I/48$ and $\wB =-J/1728$.  Let $E_t$ be the family of
elliptic curves directly $4$-congruent to $E$, as given in
Lemma~\ref{lem:4congr}.

\begin{Remark} 
\label{xe4viainvthy}
(i) The genus $1$ curves $C_t$ and $D_t$ have Jacobian $E_t$. As observed
in \cite{g1hessians}, this gives an alternative proof of
Lemma~\ref{lem:4congr}. \\
(ii) The family of quartics $G_t(X,Z)$ has constant (meaning
  independent of $t$) level $2$ theta group. It should therefore be
  possible to write $G_t(X,Z)$ as a linear combination of the binary
  quartic $G(X,Z)$ and its Hessian
  \begin{align*} H(X,Z) = (8 a c - 3 b^2) & X^4 + (24 a d - 4 b c) X^3
    Z + (48 a e + 6 b d - 4 c^2) X^2 Z^2 \\ & + (24 b e - 4 c d) X Z^3
    + (8 c e - 3 d^2) Z^4.
\end{align*}
We find that $G_t(X,Z) = (t^3 + \wA t + \wB) \big(4 \gamma(t) G(X,Z) +
\tfrac{1}{12} H(X,Z) \big)$.
\end{Remark}

We can also use Proposition~\ref{P:4cov_trivob} to describe the family
of $4$-covers $D_t$.  Let $\alpha \in F^\times/F^{\times2}\K^\times$
such that $D = D_\alpha$.  We may compute $\alpha$ from $D$ by
evaluating the rank $1$ quadratic form~\eqref{rank1} at any point $\x
\in k^4$ where it is non-zero (in each constituent field of $F$).

\begin{Theorem}
\label{thm1}
The family of $2$-covers $D_t \to C_t$ such that each $4$-cover $D_t
\to E_t$ has the same theta group as our original $4$-cover $D \to E$,
is obtained by the procedure in Section~\ref{S:24trivial}, starting in
place of~\eqref{oldstart} with
\begin{equation}
\label{Dt-alt}
\alpha \big((t+\mu)X-(\theta t + \lambda)Z\big)
=(x_1+x_2\theta+x_3\theta^2+x_4\theta^3)^2.
\end{equation}
\end{Theorem}
\begin{proof}
  We show that the family of curves $D_t$ is identical to that
  considered in Lemma~\ref{prop:int2}.  In particular, the theta
  groups are not just constant up to isomorphism, but constant as
  subschemes of $\GL_4$.  The quartic $g_t(x) = G_t(x,1)$ has root
  $\theta_t = (\theta t + \la)/(t+\mu)$. Following~\eqref{traceq}, the
  quadric intersection obtained from~\eqref{Dt-alt} is
  \[ \tr_{F/k} \left( \frac{x^2}{\alpha(t+\mu) g'_t(\theta_t)} \right)
  = \tr_{F/k} \left( \frac{\theta_t x^2}{\alpha(t+\mu) g'_t(\theta_t)}
  \right) = 0, \] where $x = x_1+x_2\theta+x_3\theta^2+x_4\theta^3$.
  Since
\begin{equation}
\label{deriv-id}
g'_t(\theta_t) = g'(\theta) d(t) / (t+\mu)^2 
\end{equation}
this may be re-written as
\[ \tr_{F/k} \left( \frac{(t+\mu)x^2 }{\alpha g'(\theta)} \right) =
\tr_{F/k} \left( \frac{(\theta t + \lambda)x^2 }{\alpha g'(\theta)}
\right) = 0. \] The formula for the Hessian used in the proof of
Lemma~\ref{prop:int2} now shows that the fibre with $t=\infty$ has
Hessian the fibre with $t=0$. The proof is completed by noting that
the fibre with $t= \infty$ is our original $4$-cover $D=D_\alpha$.
\end{proof}

\begin{Remark}
\label{getquartic}
Expanding~\eqref{Dt-alt} in powers of $\theta$ gives $4$ equations of
the form
\[ \text{(linear form in $X$, $Z$, $tX$, $tZ$)} = \text{(quadratic
  form in $x_1, \ldots, x_4$)}. \] Eliminating $X,Z,t$, we obtain a
single quartic equation in $x_1, \ldots ,x_4$, describing the twist of
\eqref{E:quartic_surface} that has $D_t$ as fibres. This provides an
alternative way of arriving at the equation for $T^+_{E,\xi}(4)$ in
Corollary~\ref{C:theta_surface_equation}.
\end{Remark}

\subsection{The twisted Shioda modular surface of level 4}
\label{sec:twisted_shioda}
As in Section~\ref{S:Dprime_nu}, let $L = \K[\varphi]$ where $\varphi$
is a root of $x^3 + \wA x + \wB = 0$, and write $\sigma$ for the
involution of $LF$ swapping $\theta$ and $\tilde{\theta}$.  We write
$\tilde{\la} = \sigma(\la)$ and $\tilde{\mu} = \sigma(\mu)$.  By the
construction in Section~\ref{S:Dprime_nu}, alternative equations for
the curves $D_t$ in Theorem~\ref{thm1} are given by
\begin{equation*}
  \alpha \tilde{\alpha} \big( (t + \mu) X - (\theta t +   \la) Z \big)
  \big( (t + \tilde{\mu}) X - (\tilde{\theta} t + \tilde{\la}) Z \big)
  = (y_1 m_1 + \ldots + y_6 m_6)^2. 
\end{equation*}

We now modify this by introducing a factor $\nu(t)$ representing the
shift from Theorem~\ref{T:theta_shift}.

\begin{Theorem}
\label{thm2}
The family of $2$-covers $\D_t \to C_t$
such that each $4$-cover $\D_t \to E_t$ has the same fibre above $0$
(as an $E[4]$-torsor) as our original $4$-cover $D \to E$, is obtained
by the procedure in Section~\ref{S:Dprime_nu}, starting in place
of~\eqref{start} with
\begin{equation}
\label{keyeqnS4}
\alpha \tilde{\alpha} \big( (t + \mu) X - (\theta t +   \la) Z \big)
\big( (t + \tilde{\mu}) X - (\tilde{\theta} t + \tilde{\la}) Z \big)  = 
\nu(t)
(y_1 m_1 + \ldots + y_6 m_6)^2  
\end{equation}
where
\begin{equation}
\label{eqn:nu}
\nu(t) = \frac{ d(t) } {t^2 - 2 t \varphi - 2 \varphi^2 - \wA }. 
\end{equation}
\end{Theorem}
\begin{proof}
  We first note that
\begin{equation}
\label{norm-id}
N_{L/\K}(t^2 - 2 t \varphi - 2 \varphi^2 - \wA) = d(t), 
\end{equation}
and so $\nu(t)$ does indeed correspond to an element of $H^1(\K,E[2])$
via the isomorphism~\eqref{H1isom}.

In Section~\ref{sec:geomS4} we exhibited another family of $4$-covers
$\D_t \to E_t$. This had constant fibre above $0$, as could be checked
by substituting $(X:Z) = (0:1)$, $(1:0)$, $(1:1)$ or $(\la:1)$ into
the equations~\eqref{sixeqns}, and observing that in each case the $4$
solutions for $(y_1 : \ldots : y_6) \in \PP^5(\Kbar)$ do not depend on
$\la$. The argument here is similar.

A calculation using~\eqref{invIJ}, \eqref{JacBQ}, \eqref{embphi}
and~\eqref{lamu} shows that
\begin{equation}
\label{cross-id}
(t + \tilde{\mu})(\theta t + \la) - (t+\mu) (\tilde{\theta} t +
\tilde{\la}) = (\theta-\tilde{\theta}) (t^2 - 2 \varphi t - 2
\varphi^2 - \wA). 
\end{equation}
Let $\tilde{\tilde{\theta}}$ be another root of $g$.  Substituting $X=
\tilde{\tilde{\theta}} t + \tilde{\tilde{\la}}$ and $Z = t +
\tilde{\tilde{\mu}}$ into the left hand side of~(\ref{keyeqnS4}),
gives a quartic polynomial in $t$, which by~\eqref{norm-id}
and~\eqref{cross-id} is a constant times $\nu(t)$, as defined
in~\eqref{eqn:nu}. Therefore the $16$ points on $\D_t$ mapping to the
points on $C_t$ with $Y=0$, are independent of $t$.  This shows that
the fibre above $0$ is constant (with respect to $t$) as a $k$-scheme.
In Section~\ref{sec:geomS4} we saw that the action of $E_t[4]$ on
$\D_t$ is given by formulae independent of $t$. Therefore the fibre
above $0$ is also constant as an $E[4]$-torsor.

Finally we note that taking $t = \infty$ gives the cover~\eqref{start}
with $\nu = 1$, which by Remark~\ref{rems1}(ii) is isomorphic to $D$.
Of course, setting $t = \infty$ in~\eqref{eqn:nu} does not literally
make sense. However after homogenising, and rescaling by a square, we
do indeed have $\nu(\infty) = 1$.
\end{proof}

\begin{Corollary}
\label{cor1}
Suppose that $\xi \in H^1(k,E[4])$ and $D_{E,\xi}$ is given
by~\eqref{oldstart}.  Then the surface $S^+_{E,\xi}(4)$ has a singular
model in $\PP^5$ defined by $3$ quadrics. These quadrics are obtained
from the $[M:\K]=6$ equations in $\K(t)[X,Z,y_1, \ldots,y_6]$ coming
from~\eqref{keyeqnS4}, by taking linear combinations to eliminate
$X^2, XZ$ and $Z^2$.
\end{Corollary}

\begin{proof}
  The key point is that the $3$ quadrics
  are independent of $t$. Again the argument is best understood by
  comparing with the situation in Section~\ref{sec:geomS4}.  The same
  calculation as mentioned in the proof of Theorem~\ref{thm2} shows
  that the linear combinations of the left hand sides
  in~\eqref{sixeqns} that vanish at $(X:Z) = (0:1)$, $(1:0)$, $(1:1)$
  and $(\la:1)$, and therefore vanish identically, do not depend on
  $t$. This explains why the the first $3$ quadrics
  in~\eqref{fivequads} do not depend on $t$. The same idea works here.
\end{proof}

\begin{Remark}
  In Theorem~\ref{thm2} we not only made the fibre above $0$ constant
  as a $k$-scheme, we made it constant as a subscheme of $\PP^5$.  For
  this, and the application to Corollary~\ref{cor1}, we needed to know
  $\nu(t)$ mod $L^{\times 2}$, not just mod $L(t)^{\times 2}$.
\end{Remark}

The genus $1$ fibration is given either by using~\eqref{eqn:Y} (which
gives two further quadratic forms in $y_1, \ldots, y_6$, now depending
on $t$) or by using Remark~\ref{get-fib}.

\section{Computing twists of $S(4)$ and $T(4)$ in the reverse case}
\label{S:rev}

In this section we take $\xi\in H^1(\K,E[4])$ with $\ob_{E,4}(\xi)=0$
and compute models for $S^-_{E,\xi}(4)$ and $T^-_{E,\xi}(4)$.

Let $E/k$ be an elliptic curve $y^2=f(x)=x^3+Ax+B$, and
$\Delta=-16(4A^3+27B^2)$ its discriminant.  We write $E^\Delta$ for
the quadratic twist of $E$ by $\Delta$. Likewise if $C\to\PP^1$ is a
double cover, then we write $C^\Delta$ for its quadratic twist (over
$\PP^1$) by $\Delta$.  We may summarise the results of this section by
saying that everything carries over from Section~\ref{S:twist_shioda}
with the following changes. 
\begin{itemize}
\item We replace $E_t$ by $E_t^\Delta$ and $C_t$ by $C_t^\Delta$.
\item The family $D_t$ is computed using the contravariants instead of
  the covariants (these were the identity map and the Hessian).
\item In Theorem~\ref{thm1} we multiply one side of the equation by
  $g'(\theta)$. This is an element of $F$ whose norm is $\Delta$ (up
  to squares).
\item In Theorem~\ref{thm2} we multiply one side of the equation by
  $(\theta - \tilde{\theta})^2 \Delta$.
\end{itemize}

We now go through the changes in detail.

\subsection{Reverse twists of the universal elliptic curve}
Let \[E^\Delta\colon y^2=\Delta\, f(x)\] be the quadratic twist of $E$
by $\Delta$. Then by \cite{BD}*{Corollary 7.4}, or
Remark~\ref{reverse4_via_inv_thy} below, there is a reverse
$4$-congruence $\sigma\colon E[4]\to E^\Delta[4]$.  We may thus
identify $X^-_E(4)=X^+_{E^\Delta}(4)$.  It is immediate from
Lemma~\ref{lem:4congr} that
\[ E_t^\Delta : \quad y^2 = -\Delta\,f(t)\, (x - \gamma(t)) (x^3 + \wA
x + \wB), \] with base point $(x,y)=(\gamma(t),0)$, is the universal
elliptic curve over $X^-_E(4)$.

\subsection{Reverse twists of the theta modular surface} 
\label{S:reverse_theta_modular}
Let $\xi\in H^1(\K,E[4])$ with $\ob_{E,4}(\xi)=0$.  Then the $4$-cover
$D_{E,\xi}$ has a model as a quadric intersection $D \subset \PP^3$
with theta group $\Theta=\Theta_{E,\xi}\subset\GL_4$. Let $\Theta^\vee
\subset \GL_4$ be the subgroup of matrices inverse transpose to those
in $\Theta$. Let $\pi : \Theta^\vee \to E^\Delta[4]$ be the map that
makes the following diagram commute (where the superscript $-T$
denotes inverse transpose)
\[\xymatrix{
  0\ar[r]&\Gm\ar[r]\ar[d]^{.^{-1}}&\Theta\ar[r]\ar[d]^{.^{-T}}
  &E[4]\ar[r]\ar[d]^{\sigma}&0\\
  0\ar[r]&\Gm\ar[r]&\Theta^\vee\ar[r]^-{\pi}&E^\Delta[4]\ar[r]&0 }
\]
Since $\sigma$ reverses the Weil pairing, the second row gives
$\Theta^\vee$ the structure of theta group for $E^\Delta[4]$. The
diagram then shows that $\Theta$ and $\Theta^\vee$ are
$\sigma$-isomorphic (see Definition~\ref{def:sigma-isom}).

We have $\Theta^\vee = \Theta_{E^\Delta,\xi'}$ for some $\xi'\in
H^1(\K,E^\Delta[4])$.  By \cite{descI}*{Theorem 5.2} there is a unique
model for the $4$-cover $D_{E^\Delta,\xi'}$ as a quadric intersection
$D^\vee \subset \PP^3$ with theta group $\Theta^\vee$.  The
contravariants, introduced in \cite{g1hessians}, give a way of
computing equations for $D^\vee$.  The details are as follows.

Let $D = \{ Q_1 = Q_2 = 0 \} \subset \PP^3$, and let $A_1$ and $A_2$
be the corresponding symmetric matrices.  Let $a,b,c,d,e$ and $I,J$
and $\wA, \wB$ be as defined in Section~\ref{sec:4cover}. Let $S_0,
\ldots, S_3$ be the quadratic forms corresponding to the matrices
$C_0, \ldots, C_3$ defined by
\[\adj(X A_1+Z A_2) =  C_0 X^3+  C_1 X^2Z + C_2 XZ^2 + C_3 Z^3. \]
The contravariants are defined by
\begin{align*}
R_1 & = \frac{1}{12} \left(
  \frac{\partial I}{\partial a} S_0 
+ \frac{\partial I}{\partial b} S_1
+ \frac{\partial I}{\partial c} S_2
+ \frac{\partial I}{\partial d} S_3 \right),  \\
R_2 & = \frac{1}{12} \left(
  \frac{\partial I}{\partial b} S_0 
+ \frac{\partial I}{\partial c} S_1
+ \frac{\partial I}{\partial d} S_2
+ \frac{\partial I}{\partial e} S_3 \right),  \\
R_1' & = \frac{1}{12^2} \left(
  \frac{\partial J}{\partial a} S_0 
+ \frac{\partial J}{\partial b} S_1
+ \frac{\partial J}{\partial c} S_2
+ \frac{\partial J}{\partial d} S_3 \right),  \\
R_2' & = \frac{1}{12^2} \left(
  \frac{\partial J}{\partial b} S_0 
+ \frac{\partial J}{\partial c} S_1
+ \frac{\partial J}{\partial d} S_2
+ \frac{\partial J}{\partial e} S_3 \right).  
\end{align*}

The quadric intersection $D^\vee$ then has equations
\begin{equation}
\label{dvee}
 D^\vee \colon \quad \{   -9 \wB R_2 + 2 \wA R'_2 =  9 \wB R_1 - 2 \wA R'_1 
 = 0 \} \subset \PP^3. 
\end{equation}
Indeed the invariant theory shows that $\Theta^\vee$ acts on $D^\vee$,
and that $D^\vee$ has Jacobian $E^\Delta$. Strictly speaking, to show
that $D^\vee$ has theta group $\Theta^\vee$, we need that the two
actions of $E^\Delta[4]$ on $D^\vee$ (one arising from the
identification $\Theta^\vee/\Gm = E^\Delta[4]$ and the other from the
structure of $D^\vee$ as a homogeneous space) agree.  Since they agree
up to an automorphism of $E^\Delta[4]$ that respects the Weil pairing,
the desired agreement may be achieved by adjusting our choice of
$\sigma$.
 
\begin{Remark}
\label{sigma_obsession}
It is convenient to fix our choice of $\sigma$ once and for
all. Following the proof of \cite{BD}*{Corollary 7.4} we let $\sigma$
correspond to the non-trivial element in the centre of
$\GL_2(\Z/4\Z)/\{\pm 1\}$.  One way of seeing that this $\sigma$ and
the one from the previous paragraph agree (up to sign) is by observing
that both are defined for the elliptic curve $E\colon x^3 + a_2 x^2 +
a_4 x + a_6$ over $k(a_2,a_4,a_6)$, which is a sufficiently general
elliptic curve not to admit other $\sigma$. But then any elliptic
curve over $k$ can be obtained by specializing this $E$, and $\sigma$
specializes with it.
\end{Remark} 

The following lemma follows from~\eqref{dvee} by direct calculation.
\begin{Lemma}
\label{L:twist_intermediate_covers}
If $D \subset \PP^3$ has intermediate $2$-cover $C\colon Y^2 = G(X,Z)$
then $D^\vee \subset \PP^3$ has intermediate $2$-cover $C^\Delta\colon
Y^2 = \Delta G(X,Z)$.
\end{Lemma}

\begin{Remark}
\label{reverse4_via_inv_thy}
Since $C^\Delta$ is a $2$-cover of $E^\Delta$, it follows that
$D^\vee$ has Jacobian $E^\Delta$. As observed in \cite{g1hessians},
this gives an alternative proof that $E$ and $E^\Delta$ are reverse
$4$-congruent.
\end{Remark}

Let $u_1, \ldots, u_4$ and $v_1, \ldots, v_4$ be a pair of bases for
$F$ as a $k$-vector space, that are dual with respect to the trace
form, i.e. $\tr_{F/k}(u_i v_j) = \delta_{ij}$.  For example we could
take $u_i = \theta^{i-1}$ and $v_j = \beta_{j-1}/g'(\theta)$ where \[
\frac{g(X)}{X - \theta} = \beta_3 X^3 + \beta_2 X^2 + \beta_1 X +
\beta_0. \]

\begin{Lemma} 
\label{lem-dualtheta}
The theta groups for the quadric intersections obtained from
\begin{equation}
\label{dual-eqn1} 
  \alpha(X - \theta Z) = (x_1 u_1 + \ldots + x_4 u_4)^2 
\end{equation}
and 
\begin{equation}
\label{dual-eqn2} 
  \frac{1}{\alpha g'(\theta)} (X - \theta Z) 
     = (x_1 v_1 + \ldots + x_4 v_4)^2, 
\end{equation}
are the inverse transpose of each other.
\end{Lemma}
\begin{proof}
  Extending our field we may assume $F = k^4$. It then suffices to
  prove the lemma in the case where $u_1, \ldots, u_4$ and $v_1,
  \ldots, v_4$ are the standard bases.

  Following~\eqref{traceq}, the quadric intersections obtained
  from~\eqref{dual-eqn1} and~\eqref{dual-eqn2} are
\[ \tr_{F/k} \left( \frac{x^2}{\alpha g'(\theta)}  \right)
= \tr_{F/k}  \left( \frac{\theta x^2 }{\alpha g'(\theta)} \right) 
  = 0, \] 
and 
\[ \tr_{F/k} \left( \alpha x^2 \right) = \tr_{F/k} \left( \alpha
  \theta x^2 \right) = 0.  \] The lemma reduces to showing that if $D
\subset \PP^3$ is defined by
\[ \sum_{i=1}^4 \xi_i x_i^2 = \sum_{i=1}^4 \xi_i \theta_i x_i^2 =
0, \] then $D^\vee \subset \PP^3$ is defined by
\[ \sum_{i=1}^4 \frac{ x_i^2 }{ \xi_i g'(\theta_i) } = \sum_{i=1}^4
\frac{ \theta_i x_i^2}{ \xi_i g'(\theta_i)} = 0. \] where $g(\theta) =
\prod_{i=1}^4 (X - \theta_i)$. This follows by direct calculation
using the contravariants.
\end{proof}

We obtain the following analogue of Theorem~\ref{thm1}. Let $\la$ and
$\mu$ be as defined in Lemma~\ref{prop:int2}.  We identify $E[4] =
E^\Delta[4]$ via $\sigma$ as specified in
Remark~\ref{sigma_obsession}.

\begin{Theorem}
\label{thm1-rev}
The family of $2$-covers $D_t^\vee \to
C_t^\Delta$
such that each $4$-cover $D_t^\vee \to E_t^\Delta$ has theta group the
inverse transpose of that for our original $4$-cover $D \to E$, is
obtained by the procedure in Section~\ref{S:24trivial}, starting in
place of~\eqref{oldstart} with
\begin{equation}
  \label{Ds-alt}
  \alpha g'(\theta)\big((t+\mu)X-(\theta t + \lambda)Z\big)
  =(x_1+x_2\theta+x_3\theta^2+x_4\theta^3)^2.
\end{equation}
\end{Theorem}

\begin{proof}
  According to Lemma~\ref{lem-dualtheta} we need to add a factor
  $g_t'(\theta_t)$ to the formula in Theorem~\ref{thm1}. However,
  since we only care about the image of this element in
  $F^\times/F^{\times2}\K^\times$ we see by~\eqref{deriv-id} that we
  can use $g'(\theta)$ instead.
\end{proof}

The analogue of Corollary~\ref{C:theta_surface_equation} is that the
surface $T^-_{E,\xi}(4)$ fibered by $D^\vee_t$ is given by
\[\{R_1R_2'-R_2R_1'=0\}\subset\PP^3.\] Alternatively, an equation for
this surface may be obtained from~\eqref{Ds-alt}, exactly as in
Remark~\ref{getquartic}.

\subsection{Reverse twists of Shioda's modular surface} 
Let $\sigma\colon E[4]\to E^\Delta[4]$ be the reverse $4$-congruence
specified in Remark~\ref{sigma_obsession}.  Given $\xi \in
H^1(k,E[4])$ with $\ob_{E,4}(\xi)=0$ we would like to write down a
model for $S^-_{E,\xi}(4)=S^+_{E^\Delta,\sigma_*(\xi)}(4)$.  If
$\ob_{E^\Delta,4}(\sigma_*(\xi))=0$, i.e.  $\sigma_*(\xi)$ is
represented by a quadric intersection (and we have these equations
explicitly) then Theorem~\ref{thm2}
gives equations for $S^-_{E,\xi}(4)$.  Unfortunately this condition is
not always satisfied.

Since $\ob_{E,4}(\xi)=0$ we have a model $D=D_{E,\xi}\subset
\PP^3$. The work in Section~\ref{S:reverse_theta_modular} gives us
$D^\vee=D_{E^\Delta,\xi'}\subset\PP^3$. It remains to determine
$\kappa=\sigma_*(\xi)-\xi'$. If $D$ is a $2$-cover of $C$ then
Lemma~\ref{L:twist_intermediate_covers} shows that $D^\vee$ is a
$2$-cover of $C^\Delta$. Since the matrix in
Remark~\ref{sigma_obsession} is congruent to the identity mod $2$, we
see that $D_{E^\Delta,\sigma_*(\xi)}$ is also a $2$-cover of
$C^\Delta$. Therefore $2\xi'=2\sigma_*(\xi)$, and so $\kappa$ is
$2$-torsion.

\begin{Lemma} 
\label{lem:kappa}
Suppose we have a $4$-cover $D=D_{E,\xi}$ with $\ob_{E,4}(\xi)=0$ and
let $D^\vee=D_{E^\Delta,\xi'}$. Then $\sigma_*(\xi)=\xi'+\kappa$,
where
\[\kappa = (3 \varphi^2 + \wA)/(4 \wA^3 + 27 \wB^2)\]
under the isomorphism~\eqref{H1isom}.
\end{Lemma}

\begin{proof}
  By the same argument as in the proof of Theorem~\ref{T:theta_shift},
  we see that $\kappa$ only depends on $\sigma\colon E[4]\to E^\Delta[4]$,
  and not on $\xi$ itself.  Thus it suffices to show that the inverse
  transpose of $\Theta_E$ is the twist of $\Theta_{E^\Delta}$ by
  $\kappa$.

  Let $E$ have Weierstrass equation $y^2 = x^3 + \wA x + \wB$. We
  embed $E \to \PP^3$ via $(x_1 : \ldots: x_4) = (1:x:y:3 x^2+ \wA)$.
  The image is $D \subset \PP^3$ defined by
  \begin{align*}
    \wA x_1^2 - x_1 x_4 + 3 x_2^2 &= 0, \\
    3 \wB x_1^2 + 2 \wA x_1 x_2 + x_2 x_4 - 3 x_3^2 &= 0.
  \end{align*}
  Then by~\eqref{dvee}, $D^\vee \subset \PP^3$ has equations
  \begin{equation}
    \label{qidual}
    \begin{aligned}
      3 \wA x_1^2 - 9 \wB x_1 x_2 - \wA^2 x_2^2
      + 6 (4 \wA^3 + 27 \wB^2) x_4^2 &= 0, \\
      9 \wB x_1^2 + 4 \wA^2 x_1 x_2 - 3 \wA \wB x_2^2 + (4 \wA^3 + 27
      \wB^2) ( x_3^2 + 4 x_2 x_4) &= 0.
    \end{aligned}
  \end{equation}

  On the other hand the $2$-cover of $E^\Delta$ corresponding to
  $\kappa$ is given by
  \[ x + (4\wA^3 + 27 \wB^2) \varphi = \kappa ( u + v \varphi + w
  \varphi^2)^2. \] Expanding and taking the coefficients of $\varphi$
  and $\varphi^2$ gives equations
\begin{equation}
\label{qivia2desc}
\begin{aligned}
  3 u^2 - 2 \wA v^2 - 4 \wA u w - 6 \wB v w + 2 \wA^2 w^2 &= 0, \\
  -4 \wA u v - 3 \wB v^2 - 6 \wB u w + 4 \wA^2 v w + 5 \wA \wB w^2 &=
  s^2.
\end{aligned}
\end{equation}
The curves~\eqref{qidual} and~\eqref{qivia2desc} are isomorphic via
\[ (x_1 : x_2 : x_3 : x_4) = ( 18 \wB v - 4 \wA^2 w : 12 \wA v + 18
\wB w : 6 s : 3 u - 2 \wA w ).\qedhere\]
\end{proof}

Using the construction in Section~\ref{S:Dprime_nu}, alternative
equations for the curves $D^\vee_t$ in Theorem~\ref{thm1-rev} are
given by
\begin{equation}
\label{alteqns}
N_{LF/M} \big(\alpha g'(\theta) \big( (t + \mu) X
- (\theta t +   \la) Z \big) \big)
= (y_1 m_1 + \ldots + y_6 m_6)^2. 
\end{equation}
We now modify this by introducing a factor $\nu(t) \kappa$
representing the shift from Theorem~\ref{T:theta_shift}.  We identify
$E[4] = E^\Delta[4]$ via $\sigma$ as specified in
Remark~\ref{sigma_obsession}.

\begin{Theorem}
  The family of $2$-covers $\D^\vee_t \to C_t^\Delta$ such that each
  $4$-cover $\D^\vee_t \to E_t^\Delta$ has the same fibre above $0$
  (as an $E[4]$-torsor) as our original $4$-cover $D \to E$, is
  obtained by the procedure in Section~\ref{S:Dprime_nu}, starting in
  place of~\eqref{start} with
  \begin{equation}
    \label{keyeqn-rev}
    N_{LF/M}(\alpha  \big( (t + \mu) X - (\theta t +   \la) Z \big))
    =  \nu(t) (\theta-\tilde{\theta})^2\Delta 
    (y_1 m_1 + \ldots + y_6 m_6)^2  
\end{equation}
where $\nu(t)$ is given by~\eqref{eqn:nu}.
\end{Theorem}

\begin{proof}
  We introduce an extra factor $\nu(t) \kappa$ to the right hand side
  of \eqref{alteqns}.  Since the factors $g'(\theta)$ and $\kappa$ do
  not depend on $t$, the proof that we obtain a family of curves with
  constant fibre above $0$ is exactly the same as for
  Theorem~\ref{thm2}.

  If $D = D_{E,\xi}$ then by Lemmas~\ref{lem-dualtheta}
  and~\ref{lem:kappa} the fibre above $t = \infty$ is $D_{E^\Delta,
    \sigma_*(\xi)}$.

  Finally we simplify our modified version of~\eqref{alteqns}.  A
  calculation along the same lines as the proof of~\eqref{cross-id}
  shows that
  \[ g'(\theta) g'(\tilde{\theta}) = -16( \theta - \tilde{\theta})^2
  (3 \varphi^2 + \wA). \] Therefore
  \[ g'(\theta) g'(\tilde{\theta}) \kappa \equiv (\theta -
  \tilde{\theta})^2 \Delta \mod{L^{\times2}}, \]
  and this gives the equation~\eqref{keyeqn-rev} as required.
\end{proof}

Exactly as in Section~\ref{sec:twisted_shioda}, we obtain the
following.

\begin{Corollary}
\label{cor-rev}
Suppose that $\xi \in H^1(k,E[4])$ and $D_{E,\xi}$ is given
by~\eqref{oldstart}.  Then the surface $S^-_{E,\xi}(4)$ has a singular
model in $\PP^5$ defined by $3$ quadrics. These quadrics are obtained
from the $[M:\K]=6$ equations in $\K(t)[X,Z,y_1, \ldots,y_6]$ coming
from~\eqref{keyeqn-rev}, by taking linear combinations to eliminate
$X^2, XZ$ and $Z^2$.
\end{Corollary}

For the purposes of Corollary~\ref{cor-rev} we may ignore the factor
$\Delta \in \K$ in~\eqref{keyeqn-rev}.  So compared to the direct
case, we only need to multiply $\nu(t)$ by a factor $(\theta -
\tilde{\theta})^2$.  This is a Kummer generator for the quadratic
extension $LF/M$.

\section{Polarizations}
\label{sec:polar}

Let $A$ be an abelian surface over a field $\K$ of characteristic $0$.
We write $A^\vee$ for the dual abelian surface.
As is described in, for instance, \cite{milne:abvar}*{Section 13}, there is an injective group homomorphism $\NS(A)\to \Hom(A,A^\vee)$. A \emph{polarization} is a homomorphism that lies in the image of the ample cone. These are isogenies. A \emph{principal polarization} is a polarization that is an isomorphism.

If an abelian variety $A$ has a principal polarization $\lambda_A$,
then the map $\lambda\mapsto\psi_\lambda=\lambda_A^{-1}\lambda$
identifies the set of polarizations with a 
special semigroup in $\End(A)$.

Elliptic curves $E$ have a natural principal polarization $\lambda_E\colon E\to E^\vee$ and on a product of elliptic curves $E\times E'$, the product of these gives a principal product polarization.

If $E, E'$ are two non-isogenous elliptic curves without complex multiplication (CM) then $\End(E\times E')=\End(E)\times\End(E')=\Z\times \Z$. 
For such a surface one has $\NS(E\times E')\simeq \Z\times\Z$, the semigroup of ample classes is $\Z_{>0}\times\Z_{>0}$, and polarizations correspond to the endomorphisms $[n]_E\times[n']_{E'}$, with $n,n'\in\Z_{>0}$.

An abelian surface $A$ is called \emph{decomposable} if it admits a non-constant map to an elliptic curve. In that case the Poincar\'e reducibility theorem
\cite{milne:abvar}*{Proposition~12.1} gives us that there are two elliptic curves $E,E'\subset A$, such that the natural map $\phi\colon E\times E'\to A$ is an isogeny. We call such an isogeny an \emph{optimal} decomposition.

In this section we are interested in determining when such a surface $A$ may admit a principal polarization $\lambda_A$. If it does, we have a polarization  $\phi^*(\lambda_A)=\phi^\vee\lambda_A\phi$ on $E\times E'$ of degree $\deg(\phi)^2$.

Write $e_n$ for the product of the Weil pairings on $(E\times E')[n]$. We paraphrase \cite{milne:abvar}*{Proposition~16.8}.

\begin{Proposition}\label{P:polarization-pullback}
  Let $\phi\colon E\times E'\to A$ be an isogeny and $\Delta = \ker \phi$.
  Let $\lambda$ be a polarization on $E\times E'$. 
  Suppose that $\Delta \subset \ker \lambda  \subset (E\times E')[n]$.
  Then $\lambda=\phi^*(\lambda')$
  for some polarization $\lambda'$ on $A$ if and only if the Weil
  pairing $e_n$ restricts to the trivial pairing on $\Delta\times
  \psi_\lambda(\frac{1}{n}\Delta)$.
\end{Proposition}

\begin{Proposition}\label{P:noprincpol}
  Let $A$ be a principally polarized decomposable abelian surface, with optimal decomposition $\phi\colon E\times E'\to A$. Suppose that $E, E'$ are non-isogenous and have no CM. Then the kernel of $\phi$ is the graph of a reverse $n$-congruence, where $\deg(\phi)=n^2$.
\end{Proposition}

\begin{proof}
Since $\ker \phi $ intersects trivially with $E\times\{0\}$ and $\{0\}\times E'$, we have that
$\ker \phi \isom \Z/d_1\Z \times \Z/d_2\Z$ for some positive
integers $d_1$ and $d_2$. 
The principal polarization on $A$ pulls back to a
polarization $\lambda$ of degree $d_1^2 d_2^2$.
It follows that $\psi_\lambda$ must be an endomorphism of the same
degree. Therefore $\psi_\lambda=[n]_E\times[n']_{E'}$ for some positive
integers $n$ and $n'$ with $n n' = d_1 d_2$. Let $\operatorname{pr}_1:
 E \times E' \to E$ be the first projection. Since
 $\ker \phi \subset \ker \lambda$ we have $\ker \phi \isom \operatorname{pr}_1(\ker \phi)
 \subset E[n]$. Therefore $d_1 \mid n$ and $d_2 \mid n$. The same argument
 shows that $d_1 \mid n'$ and $d_2 \mid n'$. Since $n n'= d_1 d_2$ it follows
 that $d_1 = d_2 = n = n'$. Hence we have that $\Delta = \ker \phi$
is the graph of an isomorphism $\sigma\colon E[n]\to E'[n]$.

  We see that
 $\Delta\subset\psi_\lambda(\frac{1}{n}\Delta)$, so
 Proposition~\ref{P:polarization-pullback} implies that
 $e_{n}(\Delta,\Delta)=1$.
In particular, if we have points $T_1=(t_1,\sigma(t_1))$ and $T_2=(t_2,\sigma(t_2))$ belonging to $\Delta$ then by Proposition~\ref{P:n-congruence_mu-auto} we have that
\[1=e_n(T_1,T_2)=e_{E[n]}(t_1,t_2)\tau_\sigma(e_{E[n]}(t_1,t_2)).\]
Therefore $\tau_\sigma(\zeta)=\zeta^{-1}$, i.e. the $n$-congruence $\sigma$ is indeed a reverse $n$-congruence.
\end{proof}

\begin{Lemma}
	\label{L:congruence_restriction}
	If $\sigma\colon E[n]\to E'[n]$ is a reverse $n$-congruence, then the restriction $\sigma'\colon E[d]\to E'[d]$ for any $d\mid n$ is also a reverse congruence.
\end{Lemma}
\begin{proof}
Let $n=dm$. If $T_1,T_2$ are generators of $E[n]$, then $mT_1,mT_2$ are generators of $E[d]$. The result follows from the basic property of Weil pairings that
\[e_{dm}(T_1,T_2)^m=e_d(mT_1,mT_2).\qedhere\]
\end{proof}

\section{Examples}
\label{S:examples}

We first give an example showing how our methods improve
on~\cite{g1hessians}. We then give an example where $\Sha(E/\Q)[4]$ is
made visible by a second elliptic curve $E'$, but our methods are
needed to find $E'$.  Finally we give some examples of $4$-torsion in
$\Sha(E/\Q)$ that cannot be made visible in a principally polarized
abelian surface.  We do this by exhibiting some twists of $S(4)$ that
are not everywhere locally soluble.

We refer to elliptic curves by their labels in Cremona's tables.

\begin{Example}
\label{ex1}
Let $E$ and $E'$ be the elliptic curves 96266a1 and 96266b1.  We have
$E(\Q) = 0$ and $E'(\Q) \isom \Z^2$. In this case there is a direct
$4$-congruence, and the Mordell-Weil group of $E'$ explains a subgroup
$(\Z/4\Z)^2 \subset \Sha(E/\Q)$.  We verify this for one element of
$\Sha(E/\Q)$ of order $4$, the other cases being similar. The element
we consider is represented by $C = \{Q_1 = Q_2=0\} \subset \PP^3$
where
\begin{align*}
  Q_1 &= x_1^2 + 3 x_1 x_2 + x_1 x_3 + x_1 x_4 - x_2^2 + 2 x_2 x_3 + 2
  x_2 x_4 -
  2 x_3^2 - x_3 x_4 - 3 x_4^2, \\
  Q_2 &= x_1^2 - 3 x_1 x_2 + 2 x_1 x_3 - 6 x_1 x_4 + 3 x_2^2 - 4 x_2
  x_3 + 2 x_2 x_4 - 4 x_3^2 - 2 x_3 x_4 - 2 x_4^2.
\end{align*} 
One of the elements of $E'(\Q)/4E'(\Q)$ of order $4$ maps to the
$4$-covering $C' = \{Q'_1 = Q'_2=0\} \subset \PP^3$ where
\begin{align*}
  Q'_1 &=
  x_1 x_2 + 2 x_1 x_3 + x_2 x_3 - x_2 x_4 + 2 x_3 x_4 + 6 x_4^2, \\
  Q'_2 &= x_1 x_2 + x_1 x_3 + 3 x_1 x_4 + 2 x_2^2 - 2 x_2 x_3 - x_2
  x_4 + 2 x_3^2 - 3 x_3 x_4 + 4 x_4^2.
\end{align*} 
Starting from either $C$ or $C'$, and computing a twist of Shioda's
modular surface using the method described in Corollary~\ref{cor1}, we
obtain
\[ S : \quad \left\{
  \begin{aligned}
    2 y_1 y_4 - 2 y_1 y_5 - 2 y_2 y_6 + y_3^2 - y_5^2 + y_6^2 &= 0 \\
    2 y_1^2 -
    2 y_1 y_3 - y_1 y_5 + y_1 y_6  - 2 y_2 y_3 \hspace{11em} & \\
    + 2 y_2 y_4 - y_2 y_5 + y_2 y_6 - y_3^2
    + y_3 y_5 + y_4 y_6 + y_5 y_6 - y_6^2 & = 0 \\
    y_1^2 + 4 y_1 y_3 - 2 y_1 y_4 - 2 y_1 y_5 + 2 y_1 y_6 + 2 y_2^2 +
    2 y_2 y_3 + 2 y_2 y_4 \quad & \\ - 2 y_2 y_5 + y_3^2 + 2 y_3 y_4 -
    2 y_3 y_6 + 2 y_4 y_5 - 2 y_5^2 + 2 y_5 y_6 + y_6^2 & = 0
\end{aligned} \right\} \subset \PP^5. \]
The embeddings $C \to S$ and $C' \to S$ are given by
\[
\begin{pmatrix}
y_{1} \\ y_{2} \\ y_{3} \\ y_{4} \\ y_{5} \\ y_{6}
\end{pmatrix} = 
\begin{pmatrix}
 30 & 16 & 12 & 22 & 10 &   8 \\
 -3 &-12 & 2 & -23 & -33 &-14 \\
 -7 & -4 & -6 & 37 & -13 & 6 \\
 41 & -4 & -38 & 21 & 3 & 6 \\
 22 & 40 & -36 & 30 & -14 & 4 \\ 
 10 & 24 & -28 & 2 & -18 & 28 
\end{pmatrix} 
\begin{pmatrix}
f_{12} \\ f_{13} \\ f_{14} \\ f_{23} \\ f_{24} \\ f_{34}
\end{pmatrix}
\]
and
\[
\begin{pmatrix}
y_{1} \\ y_{2} \\ y_{3} \\ y_{4} \\ y_{5} \\ y_{6}
\end{pmatrix} = 
\begin{pmatrix}
 -2 & -10 & -4 & 6 & 2 & 0 \\
 10 &   3 & 3 & -2 & 0 & -5 \\
-20 & 1 & -1 & 4 & 2 & 1 \\
-16 & -11 & 7 & 8 & -6 & 1 \\
-20 & -16 & -4 & 0 & -4 & -2 \\
 -8 & -16 & 4 & -2 & 2 & 0 
\end{pmatrix} 
\begin{pmatrix}
f'_{12} \\ f'_{13} \\ f'_{14} \\ f'_{23} \\ f'_{24} \\ f'_{34}
\end{pmatrix} \]
where
\[ f_{ij} = \frac{ \partial(Q_1,Q_2) }{ \partial(x_i,x_j) } \quad
\text{ and } \quad f'_{ij} = \frac{ \partial(Q'_1,Q'_2)
}{ \partial(x_i,x_j) }. \] It may be checked that these maps send the
flex points on $C$ and $C'$ to the singular points of $S$. In
particular $C$ and $C'$ correspond to the same element of
$H^1(\Q,E[4]) = H^1(\Q,E'[4])$.

Suppose instead that we use invariant theory. Let $(Q'_1,Q'_2)$ have
Hessian $(Q''_1,Q''_2)$. Then the quadric intersection $\{-281 Q'_1 +
Q''_1 = -281 Q'_2 + Q''_2 = 0\} \subset \PP^3$ is a $4$-covering of
$E$. However this $4$-covering is not locally soluble at
$2$. Therefore the method in \cite{g1hessians} for computing visible
elements of $\Sha(E/\Q)$ of order $4$ does not apply.  This is because
the shift is not locally soluble at $2$.
\end{Example}

The equations for $C$, $C'$ and $S$ in Example~\ref{ex1} were
simplified by making careful choices of coordinates. This was achieved
by a combination of {\em minimisation} and {\em reduction}.  For
quadric intersections (such as $C$ and $C'$) these processes are
described in \cite{minred234}.  We make some brief comments on how
this works for $S$.

The minimisation step relies on defining a suitable invariant.  The
discriminant of a binary cubic form $f(x,y) = a x^3 + b x^2 y + c x
y^2 + d y^3$ is
\[ \Delta(f) = -27 a^2 d^2 + 18 a b c d - 4 a c^3 - 4 b^3 d + b^2
c^2. \] Let $q_1,q_2 \in \K[z_1,z_2,z_3]$ be a pair of quadratic forms
with corresponding $3 \times 3$ symmetric matrices $B_1, B_2$. Then
$f(x,y) = \det(B_1 x + B_2 y)$ is a binary cubic form.  We define
$\Delta(q_1,q_2) = \Delta(f)$.  Now let $Q_1,Q_2,Q_3$ be quadratic
forms defining a twist of the surface $S_0$ in
Section~\ref{sec:geomS4}.  Writing $A_1,A_2,A_3$ for the corresponding
$6 \times 6$ symmetric matrices we find that
\begin{equation}
\label{fqq}
\det (A_1 z_1 + A_2 z_2 + A_3 z_3) = f(q_1,q_2)
\end{equation}
where $f$ is a binary cubic and $q_1,q_2 \in k[z_1,z_2,z_3]$ are
quadratic forms. We define $\Delta(Q_1,Q_2,Q_3) = \Delta(f)
\Delta(q_1,q_2)$.  This definition is independent of the choices of
$f, q_1, q_2$, provided that they satisfy~\eqref{fqq}.

Now let $S \subset \PP^5$ be a twist of $S_0$ defined over $\Q$.
Clearing denominators we may assume that $S$ is defined by
$Q_1,Q_2,Q_3 \in \Z[y_1, \ldots ,y_6]$. Then the discriminant $\Delta
= \Delta(Q_1,Q_2,Q_3)$ is a non-zero integer. Using the natural action
of $\GL_3(\Q) \times \GL_6(\Q)$ we seek to minimise $|\Delta|$, while
preserving that the coefficients of the $Q_i$ are integers. This
process is carried out one prime at a time, the idea being that for
each prime $p$ dividing $\Delta$ we consider the scheme defined by the
reductions of the $Q_i$ mod $p$.  We did not work out algorithms
guaranteed to minimise $|\Delta|$, but rather implemented some methods
that seem to work reasonably well in practice.

The reduction step relies on defining a suitable inner product.
Specifically we take the inner product (unique up to scalars) that is
invariant under the action of $\ASL(\Z/4\Z \times \mu_4)$. For the
surface $S_0$ in Section~\ref{sec:geomS4} this is the standard inner
product.  For general $S$ we reduce to this case by finding a change
of co-ordinates over $\C$ relating $S$ and $S_0$.  Performing lattice
reduction on the Gram matrix of the inner product then gives a change
of coordinates in $\GL_6(\Z)$ that may be used to simplify our
equations for $S$.

In preparing Example~\ref{ex1} we also had to find the change of
co-ordinates relating the surfaces constructed from $C$ and $C'$.
However it was easy to solve for this as the unique change of
co-ordinates defined over $\Q$ taking the singular points to the
singular points.

\begin{Example}
  Let $E$ be the elliptic curve 31252a1. We have $E(\Q)=0$ and
  $\Sha(E/\Q)[4] \isom (\Z/4\Z)^2$.  One of the elements of order $4$
  in $\Sha(E/\Q)$ is represented by the $4$-covering $\{ Q_1 = Q_2 = 0
  \} \subset \PP^3$ where
\begin{align*}
  Q_1 & =
  2 x_1 x_2 + 2 x_1 x_3 + x_3^2 + 2 x_4^2 \\
  Q_2 & = 6 x_1^2 + 6 x_1 x_2 - 14 x_1 x_3 + 9 x_1 x_4 + 11 x_2^2 + 10
  x_2 x_3 - 31 x_2 x_4 + 3 x_3^2 + 22 x_3 x_4 + 7 x_4^2
\end{align*}
Corollary~\ref{cor-rev} gives the following reverse twist of Shioda's
modular surface.
\[ S : \quad \left\{
  \begin{aligned}
    2 y_1 y_3 + 2 y_1 y_5 + 2 y_1 y_6 + 2 y_2^2 
        - 2 y_2 y_4 - 2 y_3 y_4 + y_5^2 & = 0 \\
    y_1^2 - 2 y_1 y_2 - y_1 y_4 + 2 y_1 y_5 + y_2^2
      + y_2 y_3 \hspace{5em} & \\  - y_2 y_4  + y_2 y_5 + y_2 y_6 - 
        y_3 y_5 + y_3 y_6 - y_4^2 - y_4 y_6 &= 0 \\
    y_1^2 - y_1 y_3 + 2 y_1 y_4 - y_1 y_5 - y_1 y_6 
    + 2 y_2 y_3 \hspace{4em} & \\ + 3 y_2 y_4 + y_3^2 + y_3 y_5 + 
        y_4 y_5 + y_4 y_6 - y_5^2 - y_5 y_6 & = 0
      \end{aligned} \right\} \subset \PP^5. \] A useful check on our
    calculations is that the flex points on $C$ and the singular
    points of $S$ have the same field of definition. The genus $1$
    fibration on $S$ may be computed as described in
    Remark~\ref{get-fib}.  We searched for rational points on $S$ of
    small height. Among the points we found were
    \begin{align*} &( 3 : 0 : -1 : -2 : -2 : 3 ), \quad
      ( 1 : -8 : 3 : -22 : -6 : 31 ), \\
      &\quad ( -33 : 13 : -25 : 23 : 34 : 22 ), \quad ( 21 : -10 : -17
      : -26 : -14 : 55 ),
    \end{align*}
    all lying on a fibre isomorphic to
    $E': y^2 = x^3 + 10609x + 58646$ with $E'(\Q) \isom \Z^3$. The
    elliptic curves $E$ and $E'$ are reverse $4$-congruent. It turns
    out that all of $\Sha(E/\Q)[4]$ is explained by $E'(\Q)$. The
    conductors of $E$ and $E'$ are $31252 = 2^2 \cdot 13 \cdot 601$
    and $2468908 = 2^2 \cdot 13 \cdot 79 \cdot 601$.  In particular
    $E'$ is beyond the range of any current tables of elliptic
    curves. (In fact $E$ is $2$-congruent to a rank 2 elliptic curve
    of the same conductor, but these curves are not $4$-congruent.)
\end{Example}

Finally we give some examples where our twists of $S(4)$ are not
locally soluble. As explained in the introduction, this can only
happen in the reverse case.  In Table~\ref{tab:insol} we list some
elliptic curves $E/\Q$ with $E(\Q) = 0$ and $\Sha(E/\Q)[4] \isom
(\Z/4\Z)^2$. In each case, for {\em some} of the elements
 $\xi \in H^1(\Q,E[4])$, representing an element of $\Sha(E/\Q)$ 
of order $4$, the surface $S_{E,\xi}^-(4)$ has no points
locally at $p$, where $p$ is the prime indicated.

\begin{table}[ht]
\label{tab:insol}
\caption{Some elliptic curves $E$ for which there exists $\xi\in H^1(\Q,E[4])$ 
with $[C_{E,\xi}] \in \Sha(E/\Q)$, yet $S^-_{E,\xi}(4)(\Q_p)=\emptyset$. Proposition~\ref{P:invis} establishes that $[C_{E,\xi}]$ is not visibile in a principally polarized abelian surface.}
$\begin{array}{|c|l|} \hline
p = 2 &  21720c1, 26712e1, 32784c1, 32816j1, 33536e1, 34560o1, 37984e1, \\ 
  & 40328b1, 47664p1, 49176b1, 59248g1, 62328bj1, 69192f1, 69312ch1, \\ & 69312dp1, 
    73600bn1, 73840a1, 74368b1, 77440cl1, 77440cr1, 77600p1 \\ \hline
    p = 5 & 23950g1, 60725j1, 63825g1, 64975e1, 72600df1, 76175e1, 90450bs1, \\
 & 105350z1, 120300n1, 121950ca1, 129850r1, 133950cy1, 137025s1, \\ & 141200bf1, 
    146700p1, 153425u1, 153425bd1, 154850m1, 154850m2\\ \hline
    p = 13 & 56446n1, 62192t1, 70135c1, 100386g1, 104442w1, 124384g1, \\
& 132496df1, 172042o1, 200772u1, 216151f1, 226629g1, 256880dn1, \\
& 294060j1, 306735z1, 321945v1, 331240cy1, 335296dj1, 337155x1 \\ \hline
    p = 29 & 220342v1, 277530bc1, 277530bs1, 323785n1, 364994k1 \\ \hline
    p = 37 & 370999a1 \\ \hline
    p = 61 & 301401k1, 260470l1, 260470l2 \\ \hline 
    p = 101 & 306030bg1, 306030bg2 \\ \hline
\end{array}$
\end{table}

\begin{Proposition}
\label{P:invis}
Each of the
elliptic curves $E/\Q$ in Table~\ref{tab:insol} has an
element
of order $4$ in $\Sha(E/\Q)$ that cannot be made visible in a
principally polarized abelian surface over $\Q$.
\end{Proposition}
\begin{proof}
  By construction, there exists $\xi \in H^1(\Q,E[4])$ such that
  $[C_{E, \xi}] \in \Sha(E/\Q)$ has order $4$, yet
  $S^-_{E,\xi}(4)(\Q_p)= \emptyset$.  Let us now assume that
  $[C_{E,\xi}]$ is visible in an abelian surface $A$, i.e., that there
  is an injection $E\to A$ such that $[C_{E,\xi}]$ lies in
  the kernel of the induced map on Galois cohomology $H^1(\Q,E)\to
  H^1(\Q,A)$.

  As described in Section~\ref{sec:polar}, there is an elliptic curve
 $E' \subset A$ and an optimal decomposition $\phi : E \times E' \to A$.
  In particular,
  the kernel of $\phi$ is the graph of an isomorphism between finite subgroups
  of $E$ and $E'$. 

  For each of the elliptic curves on our list we have $E(\Q)/2 E(\Q)=0$,
  equivalently $\rank E(\Q)= 0$ and $E[2]$ is irreducible.
  By \cite{fisher:invis}*{Theorem 3.1}, there exists,
  for some $l\geq 2$, a congruence $\sigma\colon E[2^l]\to E'[2^l]$
  such that the graph of $\sigma$ is contained in $\ker \phi$, and
  $[C_{E,\xi}] = \pi(P')$ for some $P' \in E'(\Q)$, where $\pi$ is the
  diagonal map in the following commutative diagram \[ \xymatrix{
    E(\Q)/2^l E(\Q) \ar[r] & H^1(\Q,E[2^l]) \ar[r] \ar@{=}[d] &
    H^1(\Q,E)[2^l] \ar[r] & 0 \\ E'(\Q)/2^l E'(\Q) \ar[r]
    \ar[urr]_(0.7){\pi} & H^1(\Q,E'[2^l]) \ar[r] & H^1(\Q,E')[2^l]
    \ar[r] & 0 } \]
  Since $E(\Q)/2^l E(\Q) =0$, we see that $P' \in E'(\Q)/2^l E'(\Q)$ 
  has order $4$. Since $E'(\Q)[2]=0$ it follows that $\rank E'(\Q) > 0$ 
  and $P' \in 2^{l-2} E'(\Q)$. 
  Therefore $\xi$ is explained (via a
  $4$-congruence) by an element of $E'(\Q)/4 E'(\Q)$.

  Since $\rank E(\Q)= 0$ and $\rank E'(\Q) >0$, it is clear that
  $E$ and $E'$ are not isogenous. 
  Computation shows that $\Gal(\Q^\sep/\Q)$ acts on $E[4]$ via a large
  enough group to ensure that
  $E$ has no CM.
  Since $E$ and $E'$ are $4$-congruent, this also shows that $E'$ 
  has no CM.
  
  Proposition~\ref{P:noprincpol} and Lemma~\ref{L:congruence_restriction}
  show that for $A$ to be principally polarized, the congruence $\sigma$ must be reverse.

However, as noted at the start of the proof,
  $S^-_{E,\xi}(4)$ does not have any rational points. Therefore the
  $4$-congruence induced by $\sigma$ is not reverse, and hence 
neither is $\sigma$ itself.
   This is the required contradiction.
\end{proof}

\begin{Remark} A curious fact about the examples in
  Table~\ref{tab:insol} is that the odd primes $p$ at which we find
  local obstructions satisfy $p\equiv 5\pmod{8}$. Indeed, for any one
  $p$, there are only finitely many $\Q_p$-isomorphism classes for the
  surface $S^-_{E,\xi}(4)$, so determining which ones have local
  obstructions is in principle a finite amount of
  work. Proposition~\ref{prop:5mod8} provides one description of an
  insolvability criterion, that appears to explain all the
  examples in Table~\ref{tab:insol} with $p$ odd. 
\end{Remark}

\begin{Proposition}
\label{prop:5mod8}
Let $p$ be a prime with $p \equiv 5 \pmod{8}$. Let $E$ be the 
elliptic curve $y^2 = x^3 + A x$ for some $A \in \Q_p^\times$ with $v_p(A)$ odd.
Let $\xi$ be the image of $P = (0,0)$ under the connecting map
\[   E(\Q_p)/ 4 E(\Q_p) \ra
 H^1(\Q_p,E[4]). \]
Then the surface $S_{E,\xi}^-(4)$ has no $\Q_p$-points.
\end{Proposition}

\begin{proof}
We use Corollary~\ref{cor-rev} to show that 
$S_{E,\xi}^-(4)$ has equations 
\begin{align*}
0 &= y_2 y_3 - A y_5 y_6 \\
0 &= y_1 y_4 + 2 y_2 y_5 + y_3^2 - A y_6^2 \\
0 &= (y_1^2 - 2 y_2^2) + A ( y_4^2 + 2 y_5^2 ) + 4A y_3 y_6 
\end{align*}
Since multiplying $A$ by a $4$th power gives the same elliptic
curve we may suppose $v_p(A) = \pm 1$. We consider the 
case $v_p(A) = 1$. Suppose $(y_1 : \ldots : y_6)$ is a 
$\Q_p$-point, with $y_1, \ldots, y_6 \in \Z_p$, not all in 
$p \Z_p$. Since $(2/p) = -1$ we have $y_1 \equiv y_2 \equiv y_3
\equiv 0 \pmod{p}$. Then since $(-2/p) = - 1$ we have
$y_4 \equiv y_5 \equiv y_6 \equiv 0 \pmod{p}$. This is the
required contradiction. The case $v_p(A) = -1$ is similar.
\end{proof}

\begin{Remark}
\label{rem:allinvis}
We also found 4 examples ($225336k1$, $271800bt1$, $329536y1$,
$368928bj1$) where for {\em every} $\xi \in H^1(\Q,E[4])$, representing
an element of $\Sha(E/\Q)$ of order $4$, 
the surface $S_{E,\xi}^-$ has no points locally at $2$,
and a further 4 examples ($271800bj1$, $352800md1$, $378400bv1$,
$378400by1$) where each $S_{E,\xi}^-(4)$ is locally insoluble either
at $2$ or $5$.
\end{Remark}

\begin{Example}
Let $E/\Q$ be the elliptic curve 225336k1 with
  Weierstrass equation
\[ y^2 = x^3 - x^2 - 453476 x - 197032572. \]
One of the elements of $\Sha(E/\Q)$ of order $4$ is represented by
$D=D_{E,\xi} \subset \PP^3$ with equations
\[\begin{aligned}
3 x_1 x_2 + 2 x_1 x_3 + 4 x_1 x_4 - 3 x_2 x_3 - 3 x_2 x_4 + 2 x_3^2 + x_3 x_4 + 3 x_4^2 & = 0 \\
4 x_1^2 + 2 x_1 x_2 + x_1 x_3 + x_1 x_4 + 4 x_2^2 + 2 x_2 x_3 + 4 x_3^2 - 2 x_4^2
& = 0 
\end{aligned}\]
In the reverse case we get a surface $S=S^-_{E,\xi}(4) \subset \PP^5$ with equations
\small
\[\begin{aligned}
    7 y_1 y_2 - 7 y_1 y_3 - y_1 y_4 + 19 y_1 y_5 + y_1 y_6 - 4 y_2^2 + y_2 y_3 + 5 y_2 y_4 - 
        11 y_2 y_5  \\ - 9 y_2 y_6 + 2 y_3^2 + 4 y_3 y_5 + 10 y_3 y_6 + 13 y_4^2 - 10 y_4 y_5 + 
        2 y_4 y_6 - 4 y_5^2 - 22 y_5 y_6 + 8 y_6^2 & = 0, \\
    4 y_1^2 - 8 y_1 y_3 - 6 y_1 y_4 + 11 y_1 y_5 - 18 y_1 y_6 - y_2^2 + 17 y_2 y_3 - 8 y_2 y_4
        - 11 y_2 y_5 - y_2 y_6 \\ - 5 y_3^2 + 17 y_3 y_4 + 3 y_3 y_5 + 18 y_3 y_6 - 15 y_4^2 + 
        4 y_4 y_5 - 21 y_4 y_6 - 12 y_5^2 - 13 y_5 y_6 + 3 y_6^2 & = 0, \\
    3 y_1^2 + 4 y_1 y_2 + 16 y_1 y_3 + 4 y_1 y_4 - 11 y_1 y_5 + 6 y_1 y_6 - 11 y_2^2 + 
        3 y_2 y_3 - 4 y_2 y_4 + 15 y_2 y_5 \\ + y_2 y_6 + 10 y_3^2 + 19 y_3 y_4 + 7 y_3 y_5 + 
        28 y_3 y_6 + 2 y_4^2 - 23 y_4 y_6 - y_5^2 + 31 y_5 y_6 + 20 y_6^2 & = 0.
\end{aligned}\]
\normalsize

As a check on our calculations we verified that the flex points on $D$
and the singular points on $S$ are defined over the same degree $16$
number field. 

We find that $S(\Q_2) = \emptyset$. As indicated in Remark~\ref{rem:allinvis}, 
exactly 
the same happens for the other elements of order $4$ in $\Sha(E/\Q)$. 
The argument in Proposition~\ref{P:invis} now shows that none of the elements of $\Sha(E/\Q)$ of order $4$ are
visible in a principally polarized abelian surface.
\end{Example}


\end{document}